\documentclass[10pt,a4paper]{amsart}
\usepackage[utf8]{inputenc}
\usepackage{amsmath}
\usepackage{amsfonts}
\usepackage{amssymb}
\usepackage{bbm}
\usepackage{mathrsfs}
\usepackage[all]{xy}
\usepackage[left=2cm,right=2cm,top=2cm,bottom=2cm]{geometry}
\author{Geoffrey Powell}
\title{Homological splitting results for modules over Leibniz algebras}

\address{Univ Angers, CNRS, LAREMA, SFR MATHSTIC, F-49000 Angers, France}
\email{Geoffrey.Powell@math.cnrs.fr}
\urladdr{http://math.univ-angers.fr/~powell/}
\date{}

\thanks{This work was partially supported by the ANR Project {\em ChroK}, {\tt ANR-16-CE40-0003}.}

\newtheorem{THM}{Theorem}

\newtheorem{thm}{Theorem}[section]
\newtheorem{prop}[thm]{Proposition}

\newtheorem{cor}[thm]{Corollary}
\newtheorem{lem}[thm]{Lemma}
\theoremstyle{definition}
\newtheorem{defn}[thm]{Definition}
\newtheorem{exam}[thm]{Example}

\theoremstyle{remark}
\newtheorem{rem}[thm]{Remark}
\newtheorem{nota}[thm]{Notation}

\newcommand{\kring}{\mathbbm{k}}
\newcommand{\g}{\mathfrak{g}}
\newcommand{\h}{\mathfrak{h}}
\newcommand{\glie}{{\g_{\mathrm{Lie}}}}
\newcommand{\alg}{\mathrm{Alg}}
\newcommand{\leibalg}{\alg^{\mathsf{Leib}}}
\newcommand{\liealg}{\alg^{\mathsf{Lie}}}
\newcommand{\uassalg}{\alg^{\mathsf{uAss}}}
\newcommand{\leibmod}{\mathrm{Mod}^{\mathsf{Leib}}}
\newcommand{\liemod}{\mathrm{Mod}^{\mathsf{Lie}}}
\newcommand{\uassmod}{\mathrm{Mod}}
\newcommand{\op}{^\mathrm{op}}
\newcommand{\modules}{\mathrm{Mod}}
\newcommand{\nat}{\mathbb{N}}
\renewcommand{\hom}{\mathrm{Hom}}
\newcommand{\id}{\mathrm{Id}}
\newcommand{\leibsym}{^{\mathrm{s}}}
\newcommand{\leibasym}{^{\mathrm{a}}}
\newcommand{\ext}{\mathrm{Ext}}
\newcommand{\lbs}{\mathbf{sym}}
\newcommand{\lbas}{\mathbf{asym}}
\newcommand{\ob}{\mathrm{Ob}\ }
\newcommand{\sinv}[1][]{^{\mathbf{sym}#1}}
\newcommand{\asinv}[1][]{^{\mathbf{asym}#1}}
\newcommand{\rreg}{^\downarrow}
\newcommand{\lext}[1][\mathfrak{g}]{\mathscr{E}_l^{#1}}
\newcommand{\rext}[1][\mathfrak{g}]{\mathscr{E}_r^{#1}}

\newcommand{\triv}{^{\mathrm{triv}}}
\newcommand{\kurd}{^\flat}
\newcommand{\tor}{\mathrm{Tor}}
\renewcommand{\phi}{\varphi}
\renewcommand{\epsilon}{\varepsilon}
\begin{document}

\begin{abstract}
A unified splitting result for Ext calculated in the category of modules over a Leibniz algebra is given for the case where coefficients are either both symmetric modules or both antisymmetric modules. This is a generalization of results of Loday and Pirashvili and others.
\end{abstract}

\maketitle

\section{Introduction}

Leibniz algebras are `non-commutative algebra' generalizations of Lie algebras. Given a Leibniz algebra $\g$, there is a universal surjection $\g \twoheadrightarrow \glie$ to a Lie algebra; the homological relationship between the category $\leibmod_\g$ of $\g$-modules and the category $\liemod_\glie$ of  right $\glie$-modules  is of significant interest, even when $\g = \glie$,  as in Loday and Pirashvili's work \cite{MR1383474}.  

There are two natural ways in which a right $\glie$-module $N$ can be considered as a $\g$-module, either by forming the associated {\em symmetric} $\g$-module $N\leibsym$ or the {\em antisymmetric} module $N\leibasym$. These are of fundamental importance in studying $\g$-modules: as observed by Loday and Pirashvili \cite{LP}, they allow the {\em dévissage} of the category of $\g$-modules. Explicitly, a $\g$-module $M$ fits into a natural short exact sequence 
\[
0
\rightarrow 
M_0 
\rightarrow 
M 
\rightarrow 
M/M_0 
\rightarrow 
0
\]
where $M_0$ is antisymmetric and $M/M_0$ is symmetric.

The functor $(-)\leibsym$ has a left adjoint $\lbs$ and a right adjoint $(-)\sinv$;  likewise, $(-)\leibasym$ has a left adjoint $\lbas$ and a right adjoint $(-)\asinv$. Their derived functors are important; for instance, the left derived functors $\mathbb{L}_* \lbs$ yield Leibniz homology $HL_* (\g; -)$ and the right derived functors $\mathbb{R}^*(-)\asinv$ are isomorphic to Leibniz cohomology $HL^* (\g; -)$, by the main result of \cite{LP}. 

Loday and Pirashvili \cite{MR1383474} initiated the study of the relationship between (co)homological algebra with antisymmetric coefficients and with symmetric coefficients. Feldvoss and Wagemann \cite{FW} used this to show that, for $\g$  a Leibniz algebra and $N$ a $\glie$-module,  $HL^0 (\g; N\leibasym) \cong N$ and $HL^n(\g; N\leibasym) \cong HL^{n-1}(\g ; \hom_{\kring} (\g\rreg, N)\leibsym)$ for $n>0$, where $\g\rreg$ denotes $\g$ considered as a right $\glie$-module.

This paper  generalizes techniques pioneered by Loday and Pirashvili, placing them in a unified setting. This  uses classes 
\begin{eqnarray*}
{} [\lext (N)] &\in &\ext^1_{\leibmod_\g}(N\leibsym, (N \otimes \g\rreg)\leibasym)
 \\
{} [\rext(N)]& \in &\ext^1_{\leibmod_\g}(\hom_\kring (\g\rreg, N)\leibsym, N\leibasym )
 \end{eqnarray*}
 derived from \cite{MR1383474}, where $N$ is a right $\glie$-module.

The main result is:

\begin{THM}
\label{THM}
(Theorem \ref{thm:reduction}.)
For $\g$ a Leibniz algebra and $ N_1, N_2$ right  $\glie$-modules, 
\begin{enumerate}
\item 
there is a natural isomorphism
\[
\ext^* _{\leibmod_\g}(N_1\leibsym, N_2 \leibsym )
\cong 
\ext^*_{\liemod_\glie} (N_1 , N_2) 
\oplus
\ext^{*-1} _{\leibmod_\g}((N_1 \otimes \g\rreg)\leibasym, N_2 \leibsym )
\]
and the inclusion $\ext^{*-1} _{\leibmod_\g}((N_1 \otimes \g\rreg)\leibasym,  N_2  \leibsym )
\hookrightarrow \ext^* _{\leibmod_\g}( N_1 \leibsym,  N_2  \leibsym )
$ is the Yoneda product with   $[\lext (N_1)]$;
\item 
there is a natural isomorphism
\[
\ext^* _{\leibmod_\g}( N_1 \leibasym,  N_2  \leibasym )
\cong 
\ext^*_{\liemod_\glie} (N_1 , N_2) 
\oplus
\ext^{*-1} _{\leibmod_\g}( N_1 \leibasym, \hom_\kring (\g\rreg, N_2) \leibsym )
\]
and the inclusion $\ext^{*-1} _{\leibmod_\g}( N_1 \leibasym, \hom_\kring (\g\rreg, N_2) \leibsym )
\hookrightarrow \ext^* _{\leibmod_\g}( N_1 \leibasym,  N_2  \leibasym )
$ is the Yoneda product with  $[\rext(N_2)]$.
\end{enumerate}
\end{THM}

Since Leibniz cohomology $HL^*(\g; M)$ with coefficients in a $\g$-module $M$ is isomorphic to $
\ext^* _{\leibmod_\g}( (U(\glie) \leibasym, M )$ by \cite{LP}, taking $N_1 = U(\glie)$ in the second isomorphism recovers the aforementioned isomorphism for Leibniz cohomology.

Theorem \ref{THM}  has the following conceptual interpretation: the extension classes $[\lext (N)]$ and $[\rext(N)]$ induce split distinguished triangles in the derived category of right $\glie$-modules:
\begin{eqnarray*}
&& N 
\rightarrow 
\mathbb{L}\lbs (N^s)
\rightarrow 
\mathbb{L}\lbs  ((N \otimes \g\rreg)\leibasym) [1]
\rightarrow 
\\
&&
\mathbb{R} (-)\asinv\big(\hom_\kring (\g \rreg, N) \leibsym \big)[-1]
 \rightarrow 
 \mathbb{R} (-)\asinv(N \leibasym)
  \rightarrow
 N
\rightarrow ,
\end{eqnarray*}
using the isomorphisms $\mathbb{L}\lbs (\lext(N)) \cong N$ and $\mathbb{R} (-)\asinv(\rext (N)) \cong N$ given by Propositions \ref{prop:lext_deriv_lbs} and \ref{prop:rext_deriv_asinv} respectively. 

There is an analogous result for $\tor$, which is defined using the fact that the category $\leibmod_\g$ is equivalent to the category of right $UL(\g)$-modules, where $UL(\g)$ is the enveloping algebra constructed by Loday and Pirashvili \cite{LP}. Hence one can consider $-\otimes_{UL(\g)} - $ as a bifunctor on the product of the categories of right and left $UL(\g)$-modules, together with its derived functors. 

The functor $(-)\leibsym$ corresponds to restriction along the morphism of algebras $d_1 : UL(\g)\rightarrow U(\glie)$ introduced in \cite{LP}; likewise $(-)\leibasym$ corresponds to restriction along $d_0 : UL(\g)\rightarrow U(\glie)$ (this material is reviewed in  Sections \ref{sect:background} and \ref{sect:sym_asym}). The following statement  corresponds to taking both coefficients to be symmetric:

\begin{THM} 
\label{THM:tor}
(Theorem \ref{thm:tor}.) 
For $\g$ a Leibniz algebra, $M$ a right $\glie$-module and $N$ a left $\glie$-module, there is a natural isomorphism:
\[
\tor^{UL(\g)}_* (M^{d_1}, {}^{d_1}N) 
\cong 
\tor^{U(\glie)}_* (M, N)
\oplus 
 \tor^{UL(\g)}_{*-1} ((M\otimes \g\rreg)^{d_0}, {}^{d_1}N) 
\] 
where the morphism $\tor^{UL(\g)}_* (M^{d_1}, {}^{d_1}N)  \rightarrow \tor^{UL(\g)}_{*-1} ((M\otimes \g\rreg)^{d_0}, {}^{d_1}N) $ is induced by cap product with the class $[\lext (M)]$.
\end{THM}

The associated result for antisymmetric coefficients is deduced from Theorem \ref{THM:tor} by using the involution 
$\chi_L : UL(\g)\op \stackrel{\cong}{\rightarrow} UL (\g)$ exhibited by Kurdiani \cite{MR1709257}. This induces an equivalence of categories between right $UL(\g)$-modules and left $UL(\g)$-modules that exchanges the rôles of  $d_0$ and $d_1$.

The proofs of the main results rely upon explicit models for the derived functors $\mathbb{L} \lbs$ and $\mathbb{R}(-)\asinv$ together with $\mathbb{L} \lbas$ and $\mathbb{R}(-)\sinv$. These are obtained from the Leibniz complex equipped with an explicit $\glie$-action going back to the foundational work of Loday and Pirashvili \cite{LP}; the Kurdiani involution is used to transport these  to resolutions in `left modules'.

\subsection{Conventions and Notation}

\begin{enumerate}
\item 
Throughout, $\kring$ is taken to be a field; all tensor products $\otimes$ are formed over $\kring$.
\item 
For $M$ a $\kring$-vector space,  $\langle x | x \in \mathcal{I} \rangle _\kring \subset M$ denotes the subspace generated by the set of elements $\mathcal{I} \subset M$. 
\item 
Homological conventions are used: $[1]$ always denotes homological suspension (so that $[-1]$ denotes cohomological suspension).
\item 
When considered as a complex, an object of an abelian category is placed in homological degree zero.
\end{enumerate}

\subsection{Acknowledgement}
This work forms part of a  project that was inspired by questions raised by Friedrich Wagemann; the author is grateful for his interest. 

The author is  also very grateful to  Teimuraz Pirashvili both for his interest and his comments. In particular, Pirashvili alerted the author to the existence of Kurdiani's involution \cite{MR1709257}, which allowed the author's original approach to constructing resolutions for `left modules' to be replaced by a reduction to Loday and Pirashvili's work \cite{LP}.

\tableofcontents

\section{Background}
\label{sect:background}

This Section reviews Leibniz algebras and their modules.

\subsection{Leibniz algebras and their modules}
\label{subsect:alg_mod}

Much of the following material can be found in \cite{LP} and in  \cite{Feld}, to which the reader is referred for more detail. All Leibniz algebras considered here are {\em right} Leibniz algebras.

\begin{defn}
\label{defn:Leibniz}
A  Leibniz algebra over $\kring$ is a $\kring$-vector space $\g$ equipped with a product $\g \otimes \g \rightarrow \g$, $x \otimes y \mapsto xy$, that satisfies the  Leibniz relation:
\[
x(yz) = (xy)z - (xz) y.
\]
A morphism of Leibniz algebras $\g_1 \rightarrow \g_2$ is a morphism of $\kring$-vector spaces that is compatible with the respective products. The category of Leibniz algebras is denoted $\leibalg$.  
\end{defn}

A Lie algebra is a Leibniz algebra such that the product is antisymmetric; in particular, there is a fully-faithful inclusion $
\liealg 
\hookrightarrow 
\leibalg
$ of the category of Lie algebras. This has  left adjoint given by $\g \mapsto \glie$, where
$
\glie := \g / \langle x^2 | x \in \g \rangle_\kring
$. The image of $y \in \g$ under the canonical surjection of Leibniz algebras $\g \twoheadrightarrow \glie$ is written $\overline{y}$.

\begin{defn}
\label{defn:module_leibniz}
\cite{LP}
A module over a Leibniz algebra $\g$ is a $\kring$-vector space  $M$ equipped with left and right actions $
\g \otimes M  \rightarrow  M$,  
$
M \otimes \g  \rightarrow  M$ 
that satisfy the following relations, $\forall x, y \in \g, \ m \in M$:
$$
\begin{array}{llll}
m (xy) &=& (mx)y - (my)x& \ \  (MLL) \\
x (my) &=& (xm)y - (xy)m& \ \  (LML) \\
x (ym) &=&(xy)m - (xm)y& \ \ (LLM). 
\end{array}
$$
A morphism $M_1 \rightarrow M_2$ of $\g$-modules is a morphism of $\kring$-vector spaces that is compatible with the respective left and right actions.  The category of $\g$-modules is written $\leibmod_\g$.

A $\g$-module $M$ is
\begin{enumerate}
\item 
 {\bf symmetric} if $mx + xm= 0$ $\forall (m,x )  \in M \times \g$;
 \item 
 {\bf antisymmetric} if $xm=0$ $\forall (m,x )  \in M \times \g$.
\end{enumerate}
\end{defn}

\begin{nota}
For $\h$ a Lie algebra, $\liemod_\h$ denotes the category of right $\h$-modules.
\end{nota}

The following is standard  (see \cite[Section 1.10]{LP} and  \cite[Section  3]{Feld}, for example) and also serves to introduce the notation $(-)\rreg$.

\begin{prop}
\label{prop:right-restriction}
Let $\g$ be a Leibniz algebra. 
\begin{enumerate}
\item 
The restriction to the right action induces an exact restriction functor 
$
(-)\rreg : 
\leibmod_\g
\rightarrow 
\liemod_\glie. 
$
\item 
The full subcategory of antisymmetric $\g$-modules in $\leibmod_\g$ is equivalent to $\liemod_\glie$.  
\item 
The full subcategory of symmetric $\g$-modules in $\leibmod_\g$ is equivalent to $\liemod_\glie$.  
\end{enumerate}
\end{prop}

\begin{proof}
For the first statement, it suffices to show that $x^2 \in \g$ acts trivially on the right on any $\g$-module $M$. This follows from the relation 
$
m (x^2) = (mx)x - (mx)x =0
$
given by $(MLL)$.

For the second statement, if $M$ is a $\glie$-module, then equipping it with the trivial left $\g$-action, it can be considered as an antisymmetric $\g$-module. This induces the equivalence of categories, with quasi-inverse given by the restriction functor.

The symmetric case is similar, using that, if $M$ is a symmetric module, the relation  $(x^2) m + m (x^2) =0$, implies that $(x^2)m=0$, since $m (x^2)=0$, so that $M$ arises from a $\glie$-module. 
\end{proof}

\begin{exam}
\label{exam:adjoint_rep}
For $\g$ a Leibniz algebra, the structure morphism makes $\g$ into a $\g$-module, the {\em adjoint module}. Restricting to the right action, $\g\rreg$ is a module over the Lie algebra $\glie$. 
\end{exam}

\subsection{The enveloping algebra of a Leibniz algebra}
\label{subsect:ULg}

The category $\leibmod_\g$ is described in Proposition \ref{prop:ULg_modules} using the enveloping algebra of a Leibniz algebra, as introduced by Loday and Pirashvili: 

\begin{defn}
\label{defn:UL}
\cite[Definition 2.2]{LP}
Let $UL(-)$ be the functor $UL (-) : \leibalg \rightarrow \uassalg$ to  the category of unital associative algebras that is defined on a Leibniz algebra $\g$ by 
\[
UL (\g) := T (\g^l \oplus \g^r) / I
\]
for $I$  the two-sided ideal of the tensor algebra on $\g ^l \oplus \g^r$ (where $\g \cong \g^l$, $x \mapsto l_x$ and $\g \cong \g^r$, $x \mapsto r_x$) generated by the relations for $x, y \in \g$:
\begin{eqnarray*}
r_{xy} &=& [r_x , r_y] \\
l_{xy} &=& [l_x ,r_y] \\
(r_x + l_x) l_y& =& 0. 
\end{eqnarray*}
\end{defn}

\begin{prop}
\label{prop:ULg_modules}
\cite[Theorem 2.3]{LP}
For $\g$ a Leibniz algebra, the category $\leibmod_\g$ is equivalent to the category $\uassmod_{UL(\g)}$ of right $UL (\g)$-modules. In particular, $\leibmod_\g$ has enough projectives and enough injectives.
\end{prop}

Kurdiani provided the following important additional information:

\begin{prop}
\label{prop:involution}
\cite[Proposition 2.3]{MR1709257}
For $\g$ a Leibniz algebra, there is a natural involutive anti-automorphism $\chi_L : UL(\g) \op \stackrel{\cong}{\rightarrow} UL(\g)$ of associative algebras  given by $\chi_L(r_x) = -r_x$ and $\chi_L(l_x) = r_x + l_x$. 

Thus the category of left $UL(\g)$-modules is naturally equivalent to $\uassmod_{UL(\g)}$, hence to $\leibmod_\g$ also.
\end{prop}

\begin{defn}
\label{defn:s0d0d1}
(Cf. \cite[Proposition 2.5]{LP}.)
For $\g$ a Leibniz algebra, define the natural morphisms of associative algebras 
\[
\xymatrix{
UL(\g) 
\ar@<.7ex>[r]|{d_0}
\ar@<-.7ex>[r]|{d_1}
&
U(\glie)
\ar@<-.1ex>@/_1pc/[l]|{s_0}
}
\]
for $x \in \g$ and $\overline{x}$ is its image in $\glie$ by:
\begin{enumerate}
\item 
$d_0 (l_x)=0$ and $d_0 (r_x) = \overline{x}$;
\item 
$d_1 (l_x) = - \overline{x}$ and $d_1(r_x) = \overline{x}$; 
\item 
$s_0 (\overline{x}) = r_x$.
\end{enumerate}
\end{defn}

\begin{rem}
The natural morphisms $s_0$, $d_0$, $d_1$ satisfy $ d_0s_0 = \id_{U(\glie)} = d_1 s_0$. 
\end{rem}

The importance of $d_0$ and $d_1$ is through their rôle in identifying the full subcategories of (anti)symmetric modules:

\begin{defn}
\label{defn:as_sym}
For a Leibniz algebra $\g$  let 
\begin{enumerate}
\item 
$(-)\leibasym : \liemod_{\glie} \rightarrow \leibmod_\g$ denote the functor induced by restriction along $d_0$; 
\item 
$(-)\leibsym : \liemod_{\glie} \rightarrow \leibmod_\g$ denote the functor induced by restriction along $d_1$.
\end{enumerate}
\end{defn}

\begin{rem}
For $\g$ a Leibniz algebra and $N$ a $\glie$-module, $N\leibasym$ (respectively $N \leibsym$) is the associated antisymmetric (resp. symmetric) $\g$-module, via the embeddings of Proposition \ref{prop:right-restriction}.
 Similarly, restriction along $s_0$ induces the functor $
(-)\rreg : 
\leibmod_\g
\rightarrow 
\liemod_\glie. 
$
\end{rem}

\subsection{The closed tensor structure for modules over a Lie algebra}
\label{subsect:closed_tensor_lie}

Let $\h$  be a Lie algebra over the field $\kring$. The following recalls  the closed tensor structure on the category $\liemod_\h$, where the tensor product and the internal hom use the diagonal $\h$-action. For example, for $f \in \hom_\kring (N, X)$ where $N, X \in \ob \liemod_\h$, and $h \in \h$, $fh\in \hom_\kring (N, X)$ acts by $fh (n) = f(n)h - f (nh)$. 

\begin{prop}
\label{prop:closed_tensor_lie}
For $N \in \ob \liemod_\h$, there is an adjunction
\[
-\otimes N \ : \ \liemod_\h \rightleftarrows \liemod_\h \ : \ \hom_\kring (N, -), 
\]
and both the left adjoint $- \otimes N$ and the right adjoint $\hom_\kring (N,-)$ are exact. 

For $X$ an $\h$-module, the adjunction unit $\eta_{X} : X \rightarrow \hom_\kring (N , X \otimes N) $ 
sends $x$ to the map $n \mapsto x \otimes n$; the adjunction counit $\epsilon_X : \hom_\kring (N, X) \otimes N \rightarrow X$ is the evaluation map. 
\end{prop}

 \section{Symmetric and antisymmetric modules}
\label{sect:sym_asym} 
 
This Section presents the adjunctions associated to the full subcategories of symmetric (respectively antisymmetric) modules over a Leibniz algebra $\g$. 

\subsection{Induction and coinduction}
\label{subsect:ind_coind}

\begin{nota}
For $\phi : A \rightarrow B$ a morphism of associative algebras, ${}^\phi B$ (respectively $B^\phi$) denotes $B$ considered as a left (respectively right) $A$-module, where $a\in A$ acts via multiplication by $\phi(a)$.
\end{nota}

\begin{defn}
\label{defn:induct_coinduct}
(Cf. \cite[Section II.6]{MR1731415}.)
For $\phi : A \rightarrow B$ a morphism of associative, unital algebras, and $\modules_A$, $\modules_B$ the respective categories of right modules, 
\begin{enumerate}
\item 
the restriction functor $\phi^! : \modules_B \rightarrow \modules_A$ sends a $B$-module $M$ to $M$ with $a\in A$ acting via $\phi(a)$; 
\item 
the induction functor is the functor $- \otimes_A^\phi B : \modules_A \rightarrow \modules_B$, where $B$ acts via right multiplication  on the factor $B$ and $- \otimes_A^\phi B  := - \otimes_A ({}^\phi B)$; 
\item 
the coinduction functor is the functor $\hom_A (B^\phi, -) : \modules_A \rightarrow \modules_B$, where morphisms are taken 
with respect to right $A$-modules, where the right $B$ action on $ \hom_A (B^\phi, -)  $ is induced by the  left action on $B$.
\end{enumerate}
\end{defn}

\begin{prop}
\label{prop:induct_coinduct_adjoints}
\cite[Section II.6]{MR1731415}
For $\phi : A \rightarrow B$ a morphism of associative, unital algebras, induction $- \otimes_A ^\phi B$ is left adjoint to $\phi^!$ and coinduction $\hom_A (B^\phi, -)$ is right adjoint to $\phi^!$.
\end{prop}

\subsection{Adjunctions for symmetric modules}
\label{subsect:sym_adj}

The following uses the induction, restriction and coinduction functors of Definition \ref{defn:induct_coinduct}
 and the equivalences of categories between $\leibmod_\g$ and $\uassmod_{UL(\g)}$ (respectively $\liemod_{\glie}$ and  $\uassmod_{U (\glie)}$):

\begin{prop}
\label{prop:leibsym_adjoints}
For $\g$ a Leibniz algebra, the functor 
$ (-)\leibsym : \liemod_\glie \rightarrow \leibmod$ admits both left and right adjoints: 
\[
\xymatrix{
 \liemod_\glie
\ar[rr]|{(-)\leibsym}
 &&
 \leibmod_\g.
\ar@<-1ex>@/_1pc/[ll]_{\lbs}^\perp
\ar@<+1ex>@/^1pc/[ll]^{(-)\sinv}_\perp
}
\]
In particular, $(-)\leibsym$ is exact, $\lbs$ is right exact and preserves projectives and $(-)\sinv$ is left exact and preserves injectives. 

For $M$ a $\g$-module, 
\begin{enumerate}
\item 
the $\lbs \dashv (-)\leibsym$ adjunction unit $M \twoheadrightarrow (\lbs M) \leibsym$ identifies with the surjection 
$$
 M \twoheadrightarrow  M/M_0
$$
where $M_0 := \langle mx + xm \ | \ x \in \g, m \in M \rangle _\kring$ is an antisymmetric $\g$-module;
\item 
the $(-)\leibsym \dashv (-)\sinv$ adjunction counit $(M\sinv)\leibsym \rightarrow M$ identifies with the inclusion 
$$
\{ m \in M \ | \ mx + x m = 0 \  \forall x \in \g \} \hookrightarrow M.
$$
\end{enumerate}
\end{prop}

\begin{proof}
Proposition \ref{prop:induct_coinduct_adjoints} applied to the  morphism $d_1 : UL(\g) \rightarrow U(\glie)$ gives the  module category interpretation, from which the explicit descriptions can be deduced. It is instructive to check these  directly. 

For the left adjoint, see \cite[Section 1.10]{LP} and \cite[Section 3]{Feld}. For the right adjoint, the key ingredient is that $
\{ m \in M \ | \ mx + x m = 0 \  \forall x \in \g \}
$
 is a sub $\g$-module. To show this, it suffices to check that, for any $m $ in this sub $\kring$-module and $y \in \g$,  $my \in  \{ m \in M \ | \ mx + x m = 0 \ \forall x \in \g \}$;  the result for left multiplication then follows since $ym = -my$, by the hypothesis on $m$.  Hence the result follows from the  sequence of equalities:
 \begin{eqnarray*}
 (my)x + x (my)&=&(my)x + ((xm)y -(xy)m)\\
 &=&(my)x + (xm)y + m(xy) \\
 &=&(my)x + (xm)y + ((mx)y - (my) x) \\
 &=&0,
 \end{eqnarray*}
 where the second equality uses the hypothesis on $m$ to give $-(xy) m= m(xy)$ and the final  one $xm +mx =0$. 
\end{proof}

\begin{prop}
\label{prop:lbs_ULg}
For $\g$ a Leibniz algebra, there is a natural isomorphism of right $\glie$-modules  $\lbs UL(\g) \cong U(\glie)$ and, under this isomorphism, the adjunction unit 
$UL(\g) \rightarrow (\lbs UL(\g)) \leibsym$ identifies with the morphism of $\g$-modules underlying
\[
UL(\g) 
\stackrel{d_1}{\longrightarrow} 
U(\glie)\leibsym .
\]
\end{prop} 
 
 \begin{proof}
 As in Proposition \ref{prop:leibsym_adjoints}, the functor $\lbs$ corresponds to the induction functor $- \otimes_{UL(\g)}^{d_1} U (\glie)$, which sends $UL(\g)$ to $U(\glie)$. The induction/ restriction adjunction gives that the adjunction unit is induced by $d_1$.
 \end{proof}
 
\begin{cor}
\label{cor:sinv}
For $\g$ a Leibniz algebra and $M$ a $\g$-module, there is a natural isomorphism:
\[
M\sinv \cong \hom_{\leibmod_\g} (U(\glie)\leibsym, M ) .
\]
\end{cor} 
 
\begin{proof}
Since the category of $\g$-modules is equivalent to the category of right $UL(\g)$-modules, there is 
a natural isomorphism $M\sinv \cong \hom_{\leibmod_\g} (UL(\g), (M\sinv)\leibsym)$. By adjunction, the right hand side is isomorphic to $\hom_{\leibmod_\g} ( (\lbs UL(\g))\leibsym, M)$ and, by Proposition \ref{prop:lbs_ULg}, $\lbs UL(\g) \cong U(\glie)$.
\end{proof}

Proposition \ref{prop:lbs_ULg} together with Proposition \ref{prop:leibsym_adjoints} imply that $\ker d_1$ is an antisymmetric $\g$-module. The structure of $UL(\g)$ as a $\g$-module is made explicit as follows:

\begin{prop}
\label{prop:ker_d1}
For $\g$ a Leibniz algebra, the morphism $d_1$ induces a short exact sequence of $\g$-modules:
\[
0
\rightarrow 
(\g \rreg \otimes U(\glie)) \leibasym 
\rightarrow 
UL(\g) 
\stackrel{d_1} {\rightarrow} 
U(\glie)^s 
\rightarrow 
0,
\] 
where $\g \rreg \otimes U(\glie)$ is equipped with the diagonal $\glie$-module structure. 
The monomorphism $\g \rreg \otimes U(\glie) \hookrightarrow UL(\g)$ is given by 
\[
y  \otimes \Theta 
\mapsto 
s_0 (\Theta) (r_y + l_y),
\]
for $y \in \g$ and $\Theta \in U(\glie)$. 

The morphism $s_0 : U(\glie)\rightarrow UL(\g)$ induces a  splitting of the short exact sequence as right $U(\glie)$-modules
giving $UL(\g) \cong (\g \rreg \otimes U(\glie)) \oplus U(\glie) $. 
 The extension is determined with respect to this splitting by: 
\[
s_0 (\Theta) l_y = (y \otimes \Theta) - \Theta \overline{y}
\] 
for $y \in \g$ and $\Theta \in U(\glie)$. 
\end{prop} 
 
\begin{proof}
Since $r_y + l_y$ lies in the kernel of $d_1$, the given linear map takes values in $\ker d_1$. Using \cite[Proposition 2.4]{LP}, it is straightforward to show that this induces an isomorphism of $\kring$-vector spaces. 

To verify the $UL(\g)$-module structure, since $\ker d_1$ is antisymmetric, it suffices to consider the action of $r_z \in UL(\g)$ for $z \in \g$ on an element of the image of $\g \otimes U(\glie) $ in $UL(\g)$. Using the identities $r_y r_z = r_{yz} + r_z r_y$ and $l_y r_z = l_{yz} + r_z l_y$ in $UL(\g)$, one has 
\[
s_0(\Theta) (r_y+ l_y)r_z = s_0 (\Theta ) (r_{yz} + l_{yz} ) + s_0 (\Theta) r_z (l_y + r_y)
\]
and the right hand side is the image of $yz \otimes \Theta + y \otimes \Theta \overline{z}$; this is  the image of $(y \otimes \Theta) $ under the  action of $\overline{z}$ diagonally. This identifies the module structure. 

Finally, for $\Theta \in U(\glie)$, one has $s_0 (\Theta) l_y = s_0 (\Theta) (l_y + r_y) - s_0(\Theta) r_y$ in $UL(\g)$. Here $
s_0 (\Theta) (l_y + r_y)$ is the image of $y \otimes \Theta$, giving the stated result.
\end{proof}

\begin{cor}
\label{cor:ker_d1}
For $\g$ a Leibniz algebra, there is a natural isomorphism of $\g$-modules:
\[
\ker d_1 \cong \g\triv \otimes U(\glie)\leibasym,
\]
where $\g\triv$ is considered as a $\kring$-vector space.
\end{cor}

\begin{proof}
Since $U(\glie)$ is a Hopf algebra, the $\kring$-linear inclusion $\g \hookrightarrow \g \otimes U(\glie)$, $y \mapsto y \otimes 1$ induces an isomorphism of right $U(\glie)$-modules:
\[
\g\triv \otimes U(\glie) 
\stackrel{\cong}{\rightarrow} 
\g\rreg \otimes U(\glie)
\] 
where the codomain has the diagonal action. The result then follows from Proposition \ref{prop:ker_d1} on applying the functor $(-)\leibasym$.
\end{proof}

 \subsection{Adjunctions for antisymmetric modules}
 \label{subsect:anti-sym_adjunc}

\begin{prop}
\label{prop:antisym_adjunction}
For $\g$ a Leibniz algebra, 
$ (-)\leibasym : \liemod_\glie \rightarrow \leibmod$ admits both left and right adjoints: 
\[
\xymatrix{
 \liemod_\glie
\ar[rr]|{(-)\leibasym}
 &&
 \leibmod_\g.
\ar@<-1ex>@/_1pc/[ll]_{\lbas}^\perp
\ar@<+1ex>@/^1pc/[ll]^{(-)\asinv}_\perp
}
\]
In particular, $(-)\leibasym$ is exact, $\lbas$ is right exact and preserves projectives and $(-)\asinv$ is left exact and preserves injectives. 

For $M$ a $\g$-module, 
\begin{enumerate}
\item 
the $\lbas \dashv (-)\leibasym$ adjunction unit $M \twoheadrightarrow (\lbas M) \leibasym$ identifies with the surjection 
$$
 M \twoheadrightarrow  M/M_1
$$
where  $M_1=\langle xm \ | \ x \in \g, m \in M \rangle_\kring$;
\item 
the  $(-)\leibasym \dashv (-)\asinv$ adjunction counit $(M\asinv)\leibasym \rightarrow M$ identifies with the inclusion 
$$
\{ m \in M \ | \ x m = 0\  \forall x \in \g \} \hookrightarrow M.
$$
\end{enumerate}
\end{prop}

\begin{proof}
The adjunctions can be constructed using induction and coinduction associated to $d_0 : UL(\g) \rightarrow U(\glie)$, using Proposition \ref{prop:induct_coinduct_adjoints}. 

For the explicit description of $\lbas$, it suffices to show that $M_1$ is a sub $\g$-module of $M$. It is clearly stable by the left action; stability under the right action follows from the relation $(LML)$, which gives 
$(xm) y = x (my) + (xy) m$, $\forall x, y \in \g, \ m \in M$. The right hand side  lies in $M_1$, whence the result.

The verification that $
\{ m \in M \ | \ x m = 0 \  \forall x \in \g \} \hookrightarrow M
$ is a sub $\g$-module follows from the relations for a $\g$-module (see \cite[Lemma 1.3]{FW}). 
\end{proof}

\begin{rem}
\label{rem:antisymmetric}
The submodule $M_1$ is not in general symmetric. Indeed, for $x,y \in \g$, the relation $(xm)y = x(my) +(xy)m$ implies $
(xm)y + y(xm)  = x(my) +(xy)m + y (xm). 
$ 
Then the relation  $x(my) + x (ym) =0$ (cf. the relation $(ZD)$ of \cite[Section 1]{LP})  gives
$
(xm) x + x(xm) = (x^2) m
$ on taking $x=y$. This is not in general zero (see Example \ref{exam:not_sym}).
\end{rem}

\begin{exam}
\label{exam:not_sym}
Take $M = \g$, where $\g$ is the opposite to the left Leibniz algebra of \cite[Example 2.3]{Feld}; explicitly, $\g = \langle e, f \rangle$ as $\kring$-modules, with $e^2= ef= f^2=0$ and $fe = f$. Then $(e+f)^2 = f$; hence, taking $x= (e+f) $ and $m= e$, $(x^2)m = f \neq 0$.  The kernel of the corresponding surjection $M \rightarrow (\lbas M)\leibasym $ is $\langle f \rangle \subset \g$. The relations $(e+f) f=0$ and $f(e+f) =f$ show that this is not symmetric.
\end{exam}

The following is the asymmetric counterpart of Proposition \ref{prop:asym_ULg}:

\begin{prop}
\label{prop:asym_ULg}
For $\g$ a Leibniz algebra, there is a natural isomorphism of right $\glie$-modules 
$
\lbas UL(\g) 
\cong 
U(\glie)
$ and, under this isomorphism, the canonical projection $UL(\g)\twoheadrightarrow (\lbas U(\glie)) \leibasym$
 identifies with the morphism of $\g$-modules  underlying:
\[
UL(\g) 
\stackrel{d_0}{\rightarrow} 
U(\glie)\leibasym.
\]
\end{prop}

\begin{cor}
\label{cor:right_inv_identify}
For $\g$ a Leibniz algebra and $M$ a $\g$-module, there are natural isomorphisms:
\[
M\asinv 
\cong 
\hom _{\leibmod_\g} ((\lbas UL(\g))\leibasym, M ) 
\cong 
\hom_{\leibmod_\g} (U(\glie)\leibasym, M). 
\]
\end{cor}

\begin{proof}
The proof is the same as that of Corollary \ref{cor:sinv}, {\em mutatis mutandis}.
\end{proof}

\begin{rem}
The structure of the kernel of $d_0$ is given by \cite[Proposition 2.4]{LP}; from this one sees that $\ker d_0$ is neither symmetric not antisymmetric in general. In particular, $\ker d_0$ is not determined by $(\ker d_0)\rreg$, hence there is no direct analogue of Proposition \ref{prop:ker_d1} for $\ker d_0$. 
\end{rem}

\section{Duality}
\label{sect:duality}

This Section reviews duality for modules over a Leibniz algebra $\g$. Whilst this material is not necessary for the proofs of the main result, there is an underlying duality relationship of inherent interest and which is already implicit in the literature, for example in \cite{MR1383474} and \cite{MW}. The duality functor exploits the Kurdiani involution.

\subsection{Using the Kurdiani involution}
\label{subsect:using_kurd}

Recall that the Kurdiani involution $\chi_L$ of $UL(\g)$ (cf. Proposition \ref{prop:involution}) induces an equivalence between the category of right $UL(\g)$-modules and the category of left $UL(\g)$-modules.
 
 \begin{nota}
 \label{nota:kurd}
Let $\g$ be a Leibniz algebra.
\begin{enumerate}
\item
For $M$ a $\g$-module (considered as a right $UL(\g)$-module), let $M\kurd$ denote $M$ considered as a left $UL(\g)$-module by restriction along $\chi_L : UL(\g)\op \rightarrow UL(\g)$.   
\item 
For $A$ a left $UL(\g)$-module, let ${}\kurd A$ denote $A$ considered as a right $UL(\g)$-module by restriction along $\chi_L : UL(\g)\op \rightarrow UL(\g)$.   
\end{enumerate}
 \end{nota}
 
The following is clear, since $\chi_L$ is an involution: 
 
 \begin{lem}
\label{lem:kurd_flip}
 For $\g$ a Leibniz algebra, $M$ a $\g$-module and $A$ a left $UL(\g)$-module, there are natural isomorphisms:
$
 M  \cong  {}\kurd (M\kurd) $ of right $UL(\g)$-modules and 
 $A  \cong  ({}\kurd A )\kurd$ of left 
 $UL(\g)$-modules.
 \end{lem} 
 
Let $\chi$ denote the canonical involution of the Hopf algebra $U(\glie)$; the morphisms $\chi_L$ and $\chi$ are compatible via the following result, which is proved by a direct verification:

\begin{prop}
\label{prop:compat_chiL}
For $\g$ a Leibniz algebra, there is a natural commutative diagram of morphisms of associative algebras:
\[
\xymatrix{
U(\glie)\op 
\ar[r]^{s_0\op} 
\ar[d]_\chi ^\cong
&
UL(\g)\op 
\ar[r]^{d_1\op}
\ar[d]_{\chi_L}^\cong
&
U(\glie)\op 
\ar[d]^\chi_\cong
\\
U(\glie) 
\ar[r]^{s_0}
&
UL(\g) 
\ar[r]^{d_0}
&
U(\glie).
}
\]
\end{prop}

 This gives:
 
 \begin{cor}
 \label{cor:left_right_Uglie}
 For $\g$ a Leibniz algebra, the involution $\chi$ of $U(\glie)$ induces isomorphisms
 \begin{enumerate}
 \item 
 $(U(\glie)^{d_0})\kurd \stackrel{\cong}{\longrightarrow}{}^{d_1} U(\glie)$;
 \item 
 $(U(\glie)^{d_1})\kurd \stackrel{\cong}{\longrightarrow} {}^{d_0} U(\glie)$
 \end{enumerate}
 of left $UL(\g)$-modules.
 \end{cor}

 \subsection{The duality functors}
 
 Let $(-)^\sharp$ denote the $\kring$-vector space duality functor. For $M$ a right $UL(\g)$-module, $M^\sharp$ is canonically a left $UL(\g)$-module, hence one can form the associated right $UL(\g)$-module $(M^\sharp)\kurd$; this will be denoted simply $M^\sharp$, by abuse of notation. Similarly, one has the duality functor $(-)^\sharp$ for right $U(\glie)$-modules, by exploiting the conjugation $\chi$. These are compatible, by Proposition \ref{prop:compat_chiL}:
 
 \begin{prop}
 \label{prop:duality_compatibility}
For $\g$ a Leibniz algebra, duality induces exact functors:
 \begin{eqnarray*}
 (-)^\sharp &:& (\leibmod_\g)\op \rightarrow \leibmod_\g \\
 (-)^\sharp &:& (\liemod_\glie)\op \rightarrow \liemod_\glie 
 \end{eqnarray*}
 and these restrict to equivalences of categories between the full subcategories of finite-dimensional modules. 
 
 These are compatible with the functors $(-)\rreg : \leibmod_\g \rightarrow \liemod_\glie$ and $(-)\leibsym, (-)\leibasym : \liemod_\glie \rightarrow \leibmod_\g$ in the following sense: for $M \in \ob \leibmod_\g$ and $N \in \ob \liemod_\glie$, there are natural isomorphisms:
 \begin{eqnarray*}
 (M\rreg)^\sharp &\cong & (M^\sharp)\rreg \\
 (N\leibasym)^\sharp & \cong & (N^\sharp)\leibsym \\
 (N\leibsym)^\sharp & \cong & (N^\sharp)\leibasym.  
 \end{eqnarray*}
 \end{prop}

\begin{rem}
Duality switches the rôles of symmetric and antisymmetric modules. 
\end{rem}

\begin{cor}
\label{cor:nat_duality_iso}
For $M \in \ob \leibmod_\g$, there are natural isomorphisms of functors from $(\leibmod_\g)\op$ to $\liemod_\glie$:
\begin{eqnarray*}
(\lbs M)^\sharp & \rightarrow & (M^\sharp)\asinv \\
(\lbas M)^\sharp & \rightarrow & (M^\sharp)\sinv .
\end{eqnarray*}
\end{cor}

\begin{proof}
The adjunction unit $M \rightarrow (\lbs M) \leibsym$ induces $((\lbs M) \leibsym)^\sharp\rightarrow M^\sharp$ on applying duality and $((\lbs M) \leibsym)^\sharp \cong ((\lbs M)^\sharp)\leibasym$, by Proposition \ref{prop:duality_compatibility}. The adjunction $(-)\leibasym \dashv (-)\asinv$ thus gives the first natural transformation. To see that it is an isomorphism,  use that $\lbs M$ is the cokernel of $M \otimes \g \rightarrow M$ defined by $m \otimes x \mapsto xm + mx$. On applying the duality functor one obtains the map $M^\sharp \rightarrow (M \otimes \g)^\sharp \cong \hom_\kring (\g , M^\sharp)$. This sends $f \in M^\sharp$ to $ x \mapsto fx$ for $x \in \g$ (where $fx$ is given by the right action of $\g$ on $M^\sharp$). The kernel of this map is $(M^\sharp)\asinv$, by Proposition \ref{prop:antisym_adjunction}.

The antisymmetric case is treated similarly.
\end{proof}

 Proposition \ref{prop:duality_compatibility} also gives the following:

\begin{cor}
\label{cor:duality_ext}
For $\g$ a Leibniz algebra and $\g$-modules $M, N$, duality induces a natural morphism
\[
\ext^*_{\leibmod_\g} (M, N) 
\rightarrow 
\ext^*_{\leibmod_\g} (N^\sharp, M^\sharp) 
.
\]
\end{cor}

\begin{rem}
\label{rem:duality_not_iso}
The natural transformation of Corollary \ref{cor:duality_ext} is not an isomorphism in general. However, when $\g$ is  finite-dimensional over $\kring$, one can generalize the approach of Loday and Pirashvili \cite{MR1383474} as extended by Mugniery and Wagemann \cite{MW}, replacing $\ext$ by morphisms calculated in the appropriate derived category of finite-dimensional modules. In this context, duality induces an isomorphism.
\end{rem}

\section{Resolutions}
\label{sect:resolutions}

This Section presents the fundamental resolutions in right and left $UL(\g)$-modules, also making explicit the appropriate actions of $\glie$. This builds upon results of Loday and Pirashvili \cite{LP}.

\subsection{Resolutions in right $UL(\g)$-modules}
\label{subsect:resolutions_right}

The following gives the Leibniz complex that leads to the explicit definition of Leibniz (co)homology:

\begin{thm}
\label{thm:Wg}
\cite[Section 3, Theorem 3.4]{LP}
For $\g$ a Leibniz algebra, there is a  free resolution $W_\bullet (\g)$  of $U(\glie)\leibasym$ in right $UL(\g)$-modules such that $W_n (\g) := \g^{\otimes n} \otimes UL(\g)$, with  differential 
\begin{eqnarray*}
d (x_1 \otimes \ldots \otimes x_n) \otimes 1 &=& 
(x_2\otimes \ldots \otimes x_n) \otimes l_{x_1} 
+ \sum_{i=2}^n (-1)^i (x_1 \otimes \ldots \otimes \widehat{x_i} \otimes \ldots \otimes x_n) \otimes r_{x_i} 
\\
&&
+ 
\sum_{1 \leq i < j \leq n} (-1) ^{j+1} (x_1 \otimes \ldots \otimes x_i x_j \otimes \ldots  \otimes \widehat{x_j} \otimes \ldots \otimes x_n) \otimes 1. 
\end{eqnarray*}
\end{thm}

\begin{rem}
The resolution $W_\bullet (\g)$ is natural with respect to $\g$ in the following sense:  a morphism $\g \rightarrow \g'$  of Leibniz algebras induces a morphism of complexes $W_\bullet (\g) \rightarrow W_\bullet (\g')$  of right $UL(\g)$-modules, where $W_\bullet (\g')$ is considered as a $UL(\g)$-module by restriction along $UL(\g) \rightarrow UL(\g')$. 

Similar naturality statements hold for the other structures considered here.
\end{rem}

Combining Theorem \ref{thm:Wg} with Corollary \ref{cor:ker_d1} gives:

\begin{prop}
\label{prop:resolution_Uglie_sym}
For $\g$ a Leibniz algebra, there is a free resolution $V_\bullet (\g)$ of $U(\glie)\leibsym$ in right $UL(\g)$-modules that occurs in the short exact sequence of complexes:
\[
0
\rightarrow 
UL(\g) 
\rightarrow 
V_\bullet (\g)
\rightarrow 
( \g \otimes W_\bullet (\g) )[1]
\rightarrow 
0
\]
that is determined by the differential $ V_1(\g) =  \g  \otimes UL(\g) 
 \rightarrow V_0(\g) = UL(\g)$, which is the composite:
\[
 \g  \otimes
UL(\g) 
\twoheadrightarrow 
\g \otimes U(\glie)\leibasym 
\cong 
\ker d_1 
\hookrightarrow 
UL(\g)
\]
induced by  $\g \rightarrow UL(\g)$, $x \mapsto r_x + l_x$.
\end{prop}

\begin{proof}
Corollary \ref{cor:ker_d1} gives the short exact sequence of $\g$-modules:
\[
0
\rightarrow 
\g  \otimes U(\glie)\leibasym
\rightarrow 
UL(\g)
\stackrel{d_1}{\rightarrow}
U(\glie)\leibsym
\rightarrow 
0.
\]
The complex $ \g \otimes  W_\bullet (\g) $ is a free resolution of $  \g \otimes U(\glie)\leibasym $ as a right $UL(\g)$-module, by Theorem \ref{thm:Wg}. The result then follows by splicing  with the short exact sequence.
\end{proof}

\begin{rem}
The above is implicit in the proof of \cite[Proposition 2.2]{MW}.
\end{rem}

 \subsection{Resolutions in left $UL(\g)$-modules}
 \label{subsect:left}

Using the Kurdiani involution as in Section \ref{subsect:using_kurd}, one has:
 
 \begin{lem}
\label{lem:left_right_free}
 For $Z$ a $\kring$-vector space, there is a natural isomorphism of left $UL(\g)$-modules:
 \begin{eqnarray*}
 (Z \otimes UL(\g))\kurd 
 &\stackrel{\cong}{\rightarrow}&
 UL(\g) \otimes Z 
 \\
 z \otimes a & \mapsto & \chi_L(a) \otimes z.
 \end{eqnarray*}
 \end{lem}

This allows the resolution $W_\bullet (\g)$ of $U(\glie)\leibasym$ in right $UL(\g)$-modules of Theorem \ref{thm:Wg} to be transposed to give a resolution of ${}^{d_1} U(\glie)$ in left $UL(\g)$-modules and likewise for the complex $V_\bullet (\g)$ of Proposition \ref{prop:resolution_Uglie_sym}:
 
 \begin{thm}
 \label{thm:left_module_d1_d0}
 Let $\g$ be a Leibniz algebra.  
 \begin{enumerate}
 \item 
 The complex $W_\bullet (\g) \kurd$ is  a free resolution of  $^{d_1}U(\glie)$ in left $UL(\g)$-modules. Explicitly, 
 $ 
 W_n (\g)\kurd = UL(\g) \otimes \g^{\otimes n},
 $ 
 with differential determined by 
 \begin{eqnarray*}
 d (1 \otimes \langle x_1 , \ldots ,x_n \rangle)
  &=& l_{x_1} \otimes \langle x_2, \ldots , x_n\rangle + 
 \sum_{i\geq 1} (-1)^{i+1} r_{x_i} \otimes \langle  x_1 , \ldots \widehat{x_i} , \ldots, x_n \rangle + \\
&&
 \sum_{1 \leq i< j} (-1)^{j+1} 1 \otimes \langle  x_1 , \ldots, x_ix_j, \ldots,  \widehat{x_j} , \ldots, x_n \rangle.
 \end{eqnarray*}
\item 
The complex $V_\bullet (\g)\kurd$ is a free resolution 
of ${}^{d_0} U(\glie)$ in left $UL(\g)$-modules that fits into the short exact sequence of complexes:
 \[
 0
 \rightarrow 
 UL(\g)
 \rightarrow 
 V_\bullet (\g)\kurd 
 \rightarrow 
( W_\bullet (\g)\kurd \otimes \g)[1]
 \rightarrow 
 0
 \]
 that is determined by the differential from homological degree one to degree zero:
 \[
 UL(\g) \otimes  \g
 \rightarrow 
 UL(\g)
 \] 
induced by  $\g \hookrightarrow UL(\g)$, $x\mapsto l_x$.  
\end{enumerate}
 \end{thm}
 
 \subsection{The $\glie$-actions}
This Section shows that the resolutions introduced above admit natural right $\glie$-actions. This is inspired by \cite[Proposition 3.1]{LP} (see Remark \ref{rem:caveat_LP} for one crucial difference with the definitions given there). 

The following introduces the right $\glie$-action on the relevant objects:
 
\begin{lem}
\label{lem:right_action_X}
For $\g$ a Leibniz algebra and $n \in \nat$, $UL(\g)\rreg \otimes (\g\rreg)^{\otimes n}$ is a right $\glie$-module for the diagonal action.  This action commutes with the left $UL(\g)$-module structure, so that $UL(\g) \otimes \g^{\otimes n}$ acquires the structure of a  left $UL(\g) \otimes U(\glie)\op$-module.
\end{lem} 

\begin{rem}
\label{rem:right_glie_ULg}
The right action of $\glie$ on $UL(\g)$ is given by $\Phi  \otimes \overline{y}\mapsto \Phi  r_y$, for $\overline{y} \in \glie$ and $\Phi  \in UL(\g)$. This corresponds to the right $U(\glie)$-module structure on $UL(\g)$ given by restriction along $s_0 : U(\glie) \rightarrow UL(\g)$.
\end{rem}

\begin{lem}
\label{lem:glie-linear}
For $\g$ a Leibniz algebra, the morphisms of left $UL(\g)$-modules $UL (\g) \otimes \g \rightrightarrows UL(\g)$ induced respectively by $1 \otimes x \mapsto l_x$ and $1 \otimes x \mapsto r_x$, for $x \in \g$, are $\glie$-linear with respect to the action given by Lemma \ref{lem:right_action_X}.
\end{lem}

\begin{proof}
First consider the map induced by $1 \otimes x \mapsto l_x$. The $\glie$ action by $\overline{y} \in \glie$ sends $1 \otimes x$ to $r_y \otimes x + 1 \otimes xy$, which maps to $r_y l_x + l_{xy} \in UL(\g)$.  Correspondingly, the $\glie$-action on $UL(\g)$ sends $l_x$ to $l_x r_y$. The identity $l_{xy} = [l_x , r_y]$ implies that these are equal.
 
 Similarly, consider $1 \otimes x \mapsto r_x$. The image of $r_y \otimes x + 1 \otimes xy$ under this map is $r_y r_x + r_{xy}$, whereas the $\glie$-action on $UL(\g)$ sends $r_x$ to $r_x r_y$. The identity $r_{xy}= [r_x, r_y]$ implies that these are equal. 
\end{proof}

By construction, the complexes $W_\bullet (\g)\kurd $ and $V_\bullet (\g)\kurd$ admit augmentations 
in the category of left $UL(\g)$-modules:
\begin{eqnarray*}
W_\bullet (\g)\kurd & \stackrel{d_1} {\rightarrow} & U(\glie) \\
V_\bullet (\g)\kurd & \stackrel{d_0} {\rightarrow} & U(\glie)
\end{eqnarray*}
induced by $d_1$ and $d_0$ respectively. The key fact is that the natural right $\glie$-action on $U(\glie)$  lifts using the action of Lemma \ref{lem:right_action_X} to $W_\bullet(\g)\kurd$ and $V_\bullet(\g)\kurd$. 

\begin{thm}
\label{thm:glie_action_kurd}
For $\g$ a Leibniz algebra,
\begin{enumerate}
\item 
$W_\bullet (\g)\kurd$ is a complex of left $UL(\g)$, right $\glie$-modules; 
\item 
$V_\bullet (\g)\kurd$ is a complex of left $UL(\g)$, right $\glie$-modules; moreover 
  \begin{eqnarray}
  \label{eqn:V_kurd}
 0
 \rightarrow 
 UL(\g)
 \rightarrow 
 V_\bullet (\g)\kurd 
 \rightarrow 
( W_\bullet (\g)\kurd \otimes \g\rreg)[1]
 \rightarrow 
 0
 \end{eqnarray}
is a short exact sequence of complexes of left $UL(\g)$, right $\glie$-modules, where $\glie$ acts diagonally on $ W_\bullet (\g)\kurd \otimes \g\rreg$.
\end{enumerate}
There are natural isomorphisms of right $\glie$-modules:
\[
H_0 (W_\bullet (\g)\kurd) \cong U(\glie) \cong H_0 (V_\bullet (\g)\kurd)
\] 
where $U(\glie)$ is equipped with its canonical right $\glie$-module structure.
\end{thm}

\begin{proof}
The final statement follows from the identities $d_0 s_0 = \id_{U(\glie)} = d_1 s_0$ and the identification of the right $\glie$-module structure on $UL(\g)$ given in Remark \ref{rem:right_glie_ULg}. Hence, to prove the result it suffices to show that the respective differentials are $\glie$-linear. 

First consider the case of $X_\bullet (\g)\kurd$. There are three types of term to consider in the differential:
\begin{enumerate}
\item 
the term related to $l_{x_1}$; 
\item 
the terms involving right multiplication by $r_{x_i}$ on $UL(\g)$; 
\item 
the terms involving the formation of the Leibniz product $x_i x_j$. 
\end{enumerate}
The first two contributions are treated by Lemma \ref{lem:glie-linear}. The final terms are treated by reducing to establishing the commutativity of the following diagram 
\[
\xymatrix{
(\g \otimes \g) \otimes \glie 
\ar[r]
\ar[d]_{\mu \otimes \id}
&
\g \otimes \g
\ar[d]^{\mu}
\\
\g \otimes \glie 
\ar[r]
&
\g, 
}
\]
where the horizontal morphisms are given by the $\glie$-action and $\mu$ denotes the Leibniz product of $\g$.  To see that the diagram commutes, consider the respective images of $(x_1 \otimes x_2) \otimes \overline {y}$. Passing around the top of the diagram corresponds to $(x_1 \otimes x_2) \otimes \overline {y} \mapsto (x_1y \otimes x_2) + (x_1 \otimes x_2 y) \mapsto (x_1 y) x_2 + x_1 (x_2 y)$. Passing around the bottom gives $(x_1 \otimes x_2) \otimes \overline {y} \mapsto (x_1 x_2 ) y$. These are equal by  the Leibniz relation. 
 
For the complex $V_\bullet (\g)\kurd$, since the complex identifies with $( W_\bullet (\g)\kurd \otimes \g\rreg)$ (up to shift) in strictly positive degrees, equipped with the diagonal $\glie$-action, it suffices to check that the differential between degrees $1$ and $0$ is $\glie$-linear. This is the map of left $UL(\g)$-modules $UL(\g) \otimes \g \rightarrow UL(\g)$ induced by $1 \otimes x \mapsto l_x$, hence the result follows from Lemma \ref{lem:glie-linear}.
\end{proof}

The analogous result holds for the complexes of right $UL(\g)$-modules $W_\bullet (\g)$ and $V_\bullet (\g)$. This is derived from the previous case by using the Kurdiani involution.

\begin{lem}
\label{lem:left_action_X}
For $\g$ a Leibniz algebra and $n \in \nat$, $(\g\rreg)^{\otimes n}\otimes UL(\g)$ is a left $\glie$-module for the diagonal action for the left $U(\glie)$-module ${}^{s_0} UL(\g)$ and the left $\glie$-action associated to the right $\glie$-module $ (\g\rreg)^{\otimes n}$.

This action commutes with the right $UL(\g)$-module structure, so that $(\g\rreg)^{\otimes n} \otimes UL(\g)$  acquires the structure of a right $UL(\g) \otimes U(\glie)\op$-module.
\end{lem} 

\begin{rem}
Explicitly, the left $\glie$-module structure on ${}^{s_0} UL(\g)$ is given by 
\begin{eqnarray*}
\glie \otimes UL(\g) & \rightarrow & UL(\g) \\
 \overline{y} \otimes \Phi   & \mapsto & r_y \Phi ,
\end{eqnarray*}
for $\Phi  \in UL(\g)$ and $y \in \g$. 

The left action on $\g$ is given for $x \in \g$ by $\overline{y} \otimes x \mapsto - xy$, where the sign arises from the conjugation $\chi$. 
\end{rem}

\begin{thm}
\label{thm:glie_action}
For $\g$ a Leibniz algebra, via the action of Lemma \ref{lem:left_action_X}:
\begin{enumerate}
\item 
$W_\bullet (\g)$ is a complex of right $UL(\g) \otimes U(\glie)\op$-modules; 
\item 
$V_\bullet (\g)$ is a complex of right  $UL(\g) \otimes U(\glie)\op$-modules; moreover 
  \begin{eqnarray}
\label{eqn:V} 
 0
 \rightarrow 
 UL(\g)
 \rightarrow 
 V_\bullet (\g)
 \rightarrow 
(\g\rreg \otimes  W_\bullet (\g) )[1]
 \rightarrow 
 0
 \end{eqnarray}
is a short exact sequence of complexes of right $UL(\g)\otimes U(\glie)\op$-modules, where $\glie$ acts diagonally on $\g\rreg \otimes  W_\bullet (\g) $.
\end{enumerate}
There are natural isomorphisms of left  $\glie$-modules:
\[
H_0 (W_\bullet (\g)) \cong U(\glie) \cong H_0 (V_\bullet (\g))
\] 
where $U(\glie)$ is equipped with the canonical left $\glie$-module structure. 
\end{thm}

\begin{rem}
\label{rem:caveat_LP}
In \cite[Section 3]{LP}, Loday and Pirashvili work with a {\em right} action. Moreover,  they define the action of $y \in \g$ ({\em not} $\glie$) on $UL(\g)$  by left multiplication by $l_y$ (as opposed to {\em left} multiplication by $-r_y$ as above). That this is also compatible with the differential is a consequence of the relation $(r_y + l_y ) l_x =0$. However, this does not in general pass to an action of $\glie$, since $l_y$  depends upon $y$ and not only upon $\overline{y} \in \glie$ in general. 
\end{rem}

 \section{Total derived functors}
\label{sect:total} 
 
 The resolutions of Section \ref{sect:resolutions}  give models for the total left derived functors of $\lbs , \  \lbas : \leibmod_\g \rightrightarrows \liemod_\glie$ and the total right derived functors of $(-)\sinv, \ (-)\asinv : \leibmod_\g \rightrightarrows \liemod_\glie$. 
 
 \subsection{The total left derived functors}
 
 Theorems \ref{thm:left_module_d1_d0} and \ref{thm:glie_action_kurd} established that $W_\bullet (\g)\kurd$ and $V_\bullet (\g) \kurd$ are complexes of left $UL(\g)$, right $\glie$-modules such that, as left $UL(\g)$-modules:
 \begin{enumerate}
 \item 
 $W_\bullet (\g)\kurd$ is a free resolution of ${}^{d_1} U(\glie)$; 
 \item 
 $V_\bullet (\g)\kurd$ is a free resolution of ${}^{d_0} U(\glie)$.
 \end{enumerate}

 \begin{cor}
 \label{cor:total_left}
Let  $\g$ be a Leibniz algebra and $M$ a $\g$-module. 
\begin{enumerate}
\item  
 The total left derived functor $\mathbb{L}\lbs M$ 
 is represented by the complex of right $\glie$-modules   $$
 M \otimes_{UL(\g)} W_\bullet (\g)\kurd,
 $$ 
 where the right $\glie$-action is provided by that on $ W_\bullet (\g)\kurd$.
\item 
The total left derived functor $\mathbb{L}\lbas M$ is represented by the complex of right $\glie$-modules   $$
 M \otimes_{UL(\g)} V_\bullet (\g)\kurd,
 $$ 
 where the right $\glie$-action is provided by that on $V_\bullet (\g)\kurd$.
 \end{enumerate}
Moreover, the short exact sequence of complexes (\ref{eqn:V_kurd}) induces a short exact sequence of complexes in right $\glie$-modules
\[
0
\rightarrow 
M\rreg 
\rightarrow 
 M \otimes_{UL(\g)} V_\bullet (\g)\kurd
 \rightarrow 
  M \otimes_{UL(\g)} (W_\bullet (\g)\kurd \otimes \g\rreg)[1]
  \rightarrow 
  0.
\] 
 \end{cor}
 
\begin{proof}
As in Proposition \ref{prop:leibsym_adjoints}, the functor $\lbs : \leibmod_\g \rightarrow \liemod_\glie$ identifies with the induction functor $- \otimes_{UL(\g)} {}^{d_1} U(\glie)$, where the $\glie$-action is provided by the right action on $U(\glie)$. The associated derived functor is represented by $- \otimes_{UL(\g)} W_\bullet (\g)\kurd$ by  Theorems \ref{thm:left_module_d1_d0} and \ref{thm:glie_action_kurd}. 

For the antisymmetric case, as in Proposition \ref{prop:antisym_adjunction}, the functor $\lbs : \leibmod_\g \rightarrow \liemod_\glie$ identifies with the induction functor $- \otimes_{UL(\g)} {}^{d_0} U(\glie)$. The conclusion follows analogously.

The final statement is immediate.
\end{proof} 
 
 \begin{rem}
\label{rem:identify_lder_complexes}
 The complexes appearing in Corollary \ref{cor:total_left} identify explicitly as follows:
 \begin{enumerate}
 \item 
$(M \otimes_{UL(\g)} W_\bullet (\g)\kurd)_n = M\otimes \g^{\otimes n}$ with  differential $ M\otimes \g^{\otimes n} \rightarrow M \otimes \g^{\otimes n-1}$ given  for $m \in M$ and $x_i \in \g$ by:
 \begin{eqnarray*}
  d (m \otimes \langle x_1 , \ldots ,x_n \rangle)
  &=&{x_1}m \otimes \langle x_2, \ldots , x_n\rangle + 
 \sum_{i\geq 1} (-1)^{i+1} m {x_i} \otimes \langle  x_1 , \ldots \widehat{x_i} , \ldots, x_n \rangle + \\
&&
 \sum_{1 \leq i< j} (-1)^{j+1} m \otimes \langle  x_1 , \ldots, x_ix_j, \ldots,  \widehat{x_j} , \ldots, x_n \rangle.
 \end{eqnarray*} 
 \item 
$
  (M \otimes_{UL(\g)} V_\bullet (\g)\kurd)_n = 
  M\otimes \g^{\otimes n -1} \otimes \g $ for $n>0$ and $M$ for $n=0$. 
 For $n\geq 2$, the differential 
$ M\otimes \g^{\otimes n -1}  \otimes \g
 \rightarrow
  M \otimes \g^{\otimes n-2}\otimes \g
  $ is given  for $m \in M$ and $x_i \in \g$ by:
 \begin{eqnarray*}
  d (m \otimes \langle x_1 , \ldots ,x_{n-1} \rangle \otimes x_n )
  &=&{x_1}m \otimes \langle x_2, \ldots , x_{n-1}\rangle \otimes  x_n
  + 
 \sum_{i= 1}^{n-1} (-1)^{i+1} m {x_i} \otimes \langle  x_1 , \ldots \widehat{x_i} , \ldots, x_{n-1} \rangle \otimes x_n
  + \\
&&
 \sum_{1 \leq i< j\leq n-1} (-1)^{j+1} m \otimes \langle  x_1 , \ldots, x_ix_j, \ldots,  \widehat{x_j} , \ldots, x_{n-1} \rangle \otimes x_n.
 \end{eqnarray*}
 The differential $M\otimes  \g
 \rightarrow
  M$ is given by $m \otimes x \mapsto  x m$. 
\end{enumerate} 
 \end{rem}

 \begin{rem}
 The final statement of Corollary \ref{cor:total_left} can be interpreted in the  derived category of  $\glie$-modules as giving the distinguished triangle:
 \[
 M\rreg 
 \rightarrow 
 \mathbb{L} \lbas M
 \rightarrow 
 \mathbb{L} \lbs M \otimes \g\rreg [1]
 \rightarrow .
 \]
 \end{rem}

In the symmetric case, the associated derived functors identify with Leibniz homology by  \cite{LP}:

\begin{prop} 
\label{prop:hom_lder}
For $\g$ a Leibniz algebra and $M$ a $\g$-module, there are natural isomorphisms:
\[
\mathbb{L}_*\lbs M
 \cong 
\tor _*^{UL(\g)} (M ,{}^{d_1} U(\glie))
\cong 
HL_* (\g; M\kurd).
\]
\end{prop}

\subsection{The total right derived functors}

The consideration of the total right derived functors of $(-)\asinv, \  (-)\sinv : \leibmod_\g \rightrightarrows \liemod_\glie$ is analogous, using the complexes $W_\bullet (\g)$ and $V_\bullet (\g)$ of right $UL(\g)$, left $\glie$-modules provided by Theorem \ref{thm:Wg}, Proposition \ref{prop:resolution_Uglie_sym}, with   the $\glie$-action given by Theorem \ref{thm:glie_action}. In particular, as complexes of right $UL(\g)$-modules:
\begin{enumerate}
\item 
$W_\bullet (\g)$ is a free resolution of $U(\glie) \leibasym$; 
\item 
$V_\bullet (\g)$ is a free resolution of $U(\glie) \leibsym$.
\end{enumerate}

\begin{cor}
\label{cor:total_right}
Let $\g$ be a Leibniz algebra and $M$ a $\g$-module. 
\begin{enumerate}
\item 
The total right derived functor $\mathbb{R} (-)\asinv M$ is represented by the cochain complex of right $\glie$-modules:
\[
\hom_{UL(\g)} (W_\bullet(\g), M),
\]
where the right $\glie$-action is induced by the left action on $W_\bullet (\g)$. 
\item 
The total right derived functor $\mathbb{R} (-)\sinv M$ is represented by the cochain complex of right $\glie$-modules:
\[
\hom_{UL(\g)} (V_\bullet(\g), M),
\]
where the right $\glie$-action is induced by the left action on $V_\bullet (\g)$.
\end{enumerate}
Moreover, the short exact sequence of complexes (\ref{eqn:V}) induces a short exact sequence of cochain complexes of right $\glie$-modules:
\[
0
\rightarrow 
\hom_\kring (\g\rreg, \hom_{UL(\g)} (W_\bullet (\g), M ))[-1]
\rightarrow 
\hom_{UL(\g)} (V_\bullet(\g), M) 
\rightarrow 
M\rreg 
\rightarrow 
0. 
\]
\end{cor}

\begin{rem}
Interpreted in the appropriate derived category, the short exact sequence of complexes induces the distinguished triangle:
\[
\hom_\kring (\g\rreg, \mathbb{R} (-)\asinv  M) [-1]
\rightarrow 
\mathbb{R} (-)\sinv M 
\rightarrow 
M \rreg 
\rightarrow .
\]
\end{rem}

In the antisymmetric case, the associated derived functors identify with Leibniz cohomology by the main result of \cite{LP}:

\begin{prop} 
\label{prop:cohom_rder}
For $\g$ a Leibniz algebra and $M$ a $\g$-module, there are natural isomorphisms:
\[
\mathbb{R}^*(-)\asinv M
\cong 
\ext^*_{UL(\g)} ( U(\glie)\leibasym, M)
\cong 
HL^* (\g; M) .
\]
\end{prop}

\subsection{Duality for total derived functors}

The duality isomorphisms of Corollary \ref{cor:nat_duality_iso} extend to explicit duality isomorphisms for the total derived functors at the level of the representing complexes:

\begin{prop}
\label{prop:duality_complexes}
For $\g$ a Leibniz algebra and $M$ a $\g$-module, there are natural isomorphisms of cochain complexes in right $\glie$-modules:
\begin{eqnarray*}
(M \otimes_{UL(\g)} W_\bullet (\g) \kurd) ^\sharp 
& 
\stackrel{\cong}{\rightarrow}
&
\hom_{UL(\g)} (W_\bullet (\g), M^\sharp) 
\\
(M \otimes_{UL(\g)} V_\bullet (\g) \kurd) ^\sharp 
& 
\stackrel{\cong}{\rightarrow}
&
\hom_{UL(\g)} (V_\bullet (\g), M^\sharp). 
\end{eqnarray*}
\end{prop}

\begin{proof}
This is clear at the level of the underlying graded objects, by linear algebra. The verification that the differentials correspond extends the argument used in the proof of Corollary \ref{cor:nat_duality_iso} and is left to the reader.
\end{proof}

\begin{rem}
The duality isomorphisms of Proposition  \ref{prop:duality_complexes} can be viewed at the level of total derived functors as:
\begin{eqnarray*}
(\mathbb{L} \lbs M)^\sharp 
&\cong & \mathbb{R} (-)\asinv (M^\sharp) \\
(\mathbb{L} \lbas M)^\sharp 
&\cong & \mathbb{R} (-)\sinv (M^\sharp).
\end{eqnarray*}
On passage to homology, the first isomorphism gives the duality isomorphism relating Leibniz homology and Leibniz cohomology:
\[
HL_*(\g; M\kurd) ^\sharp 
\cong 
HL^* (\g; M^\sharp),
\]
by the identifications of Propositions \ref{prop:hom_lder} and \ref{prop:cohom_rder}.
\end{rem}

\subsection{Splitting results}
Working with the explicit resolutions gives direct access to splitting results when coefficients are taken to be (anti)symmetric, as appropriate.

\begin{prop}
\label{prop:lder_sym_asym}
For $\g$ a Leibniz algebra and $N$ a right $\glie$-module,  there are natural isomorphisms of complexes of right $\glie$-modules:
\begin{eqnarray*}
 N\leibsym \otimes_{UL(\g)} W_\bullet (\g)\kurd
 &\cong &
 N
\  \oplus \ 
 \big(
 (N \otimes \g\rreg)\leibasym \otimes_{UL(\g)} W_\bullet (\g)\kurd
 \big)[1]
 \\
  N\leibasym \otimes_{UL(\g)} V_\bullet (\g)\kurd
 & \cong& 
  N 
\   \oplus \ 
(N\leibasym \otimes _{UL(\g)} W_\bullet (\g)\kurd )\otimes \g\rreg [1].  
\end{eqnarray*}
\end{prop}

\begin{proof}
Consider the complex $ N\leibsym \otimes_{UL(\g)} W_\bullet (\g)\kurd$.
The differential between homological degree $1$ and $0$ is zero, since $N\leibsym$ is symmetric, which gives the splitting of complexes. It remains to identify the complex of terms of positive degree. 

The respective terms of the differential made explicit in  Remark \ref{rem:identify_lder_complexes}  with $x_1 m$ and $mx_1$ in the first factor sum to zero. 
Hence, by inspection, one has 
\[
 d (m \otimes \langle x_1 , \ldots ,x_n \rangle)
= 
- d ((m \otimes x_1) \otimes \langle x_2 , \ldots ,x_n \rangle)
\]
where the left hand side is calculated in $ N\leibsym \otimes_{UL(\g)} W_\bullet (\g)\kurd$ and the right hand side in  $(N  \otimes \g\rreg)\leibasym \otimes_{UL(\g)} W_\bullet (\g)\kurd$. This gives the required isomorphism of complexes. 

The antisymmetric case is more direct, since the short exact sequence of complexes given in Corollary \ref{cor:total_left} splits in this case.
\end{proof}

\begin{rem}
\label{rem:split_lder}
Working in the appropriate derived category of right $\glie$-modules, these splittings yield the natural  isomorphisms of total derived functors:
\begin{eqnarray*}
\mathbb{L} \lbs (N^s) &\cong & N \ \oplus \ \mathbb{L} \lbs ((N \otimes \g\rreg)\leibasym) [1] \\
\mathbb{L} \lbas (N^a) & \cong & N\  \oplus \ \mathbb{L} \lbs (N^a) \otimes \g\rreg [1].
\end{eqnarray*}
\end{rem}

\begin{exam}
\label{exam:HL_hom_splitting}
On passage to homology, for $N \in \ob \liemod_\glie$, one obtains the isomorphism
\[
HL_* (\g; (N\leibsym)\kurd) 
\cong 
N 
\oplus 
HL_{*-1} (\g; ((N \otimes \g\rreg)\leibasym)\kurd ).
\]
\end{exam}

The following is the counterpart of Proposition \ref{prop:lder_sym_asym} for right derived functors:

\begin{prop}
\label{prop:rder_asinv_sinv}
For $\g$ a Leibniz algebra and $N$ a right $\glie$-module, there are natural isomorphisms of cochain complexes of right $\glie$-modules:
\begin{eqnarray*}
\hom_{UL(\g)} (W_\bullet (\g), N\leibasym ) 
&\cong 
&
N  \ \oplus \ \hom _{UL(\g)} (W_\bullet (\g), \hom_\kring (\g \rreg , N)\leibsym) [-1] 
\\
\hom_{UL(\g)} (V_\bullet (\g), N\leibsym ) 
&\cong 
&
N\  \oplus \ 
\hom_\kring (\g \rreg , \hom _{UL(\g)} (W_\bullet (\g), N\leibsym)) [-1].
\end{eqnarray*}
\end{prop}

\begin{proof}
The first statement is established in the proof of \cite[Lemma 1.4(b)]{FW}. The second is an immediate consequence of the splitting of the short exact sequence given in Corollary \ref{cor:total_right}. 
\end{proof}

\begin{rem}
\label{rem:split_rder}
Working in the appropriate derived category, the isomorphisms of cochain complexes given in Proposition \ref{prop:rder_asinv_sinv} induce the natural isomorphisms:
\begin{eqnarray*}
\mathbb{R} (-)\asinv (N\leibasym) &\cong & N \ \oplus \  \mathbb{R} (-)\asinv (\hom_\kring (\g\rreg, N)\leibsym ) [-1] \\
\mathbb{R} (-)\sinv (N\leibsym) &\cong & N \ \oplus \  \hom_\kring (\g\rreg, \mathbb{R} (-)\asinv  (N\leibsym) ) [-1].
\end{eqnarray*}
\end{rem}

\begin{exam}
\label{exam:splitting_HL_cohom_antisymmetric}
On passing to cohomology, for $N \in \ob \liemod_\glie$ one recovers the isomorphism
\[
HL^* (\g; N\leibasym) 
\cong 
\left\{
\begin{array}{ll}
N& *=0 \\
 HL^{*-1} (\g; \hom_\kring (\g\rreg, N) \leibsym)
 & *>0
 \end{array}
 \right.
\]
of \cite[Lemma 1.4(b)]{FW}.
\end{exam}

\section{Extensions à la Loday-Pirashvili}
\label{sect:extensions_LP}

 The key to the (co)homological splitting isomorphisms is provided by  universal extensions inspired from Loday and Pirashvili's extension  \cite[equation (1.1)]{MR1383474}.

\begin{prop}
\label{prop:LP_ses}
For $\g$ a Leibniz algebra, there are exact functors 
\[
\lext, \rext : \liemod_\glie \rightrightarrows \leibmod_\g 
\]
such that, for $N \in \ob \liemod_\glie$:
\begin{enumerate}
\item 
$\rext (N)\rreg \cong N \oplus \hom_\kring (\g\rreg, N) $, where $\hom_\kring (\g\rreg, N)$ is equipped with the diagonal structure, and there is a natural short exact sequence 
\[
0
\rightarrow 
N\leibasym
\rightarrow 
\rext (N)
\rightarrow 
\hom_\kring (\g\rreg, N) \leibsym
\rightarrow 
0,
\]  
where the left action  $\g \otimes \rext (N) \rightarrow \rext (N)$ is determined by the component $\g \otimes \hom_\kring (\g , N) \rightarrow N$, which is given by the evaluation map; 
\item 
$\lext (N)\rreg \cong (N \otimes \g\rreg) \oplus N $, where $ (N \otimes \g\rreg)$ is equipped with the diagonal structure, and there is a natural short exact sequence 
\[
0
\rightarrow 
(N \otimes \g\rreg)\leibasym
\rightarrow 
\lext (N)
\rightarrow 
N\leibsym
\rightarrow 
0,
\]  
where the left action  $\g \otimes \lext (N)\rightarrow \lext (N)$ is determined by the component $\g \otimes N
\rightarrow (N \otimes \g)$, which is given by the transposition of tensor factors. 
\end{enumerate}
Moreover, there is a natural duality isomorphism 
$ 
\lext(N)^\sharp 
\cong 
\rext (N^\sharp)$.
\end{prop} 

\begin{proof}
The first statement corresponds to the extension given in \cite[Equation (1.1)]{MR1383474}. (The reader will check that this holds for $\g$ an arbitrary Leibniz algebra, although  it is only applied to the case $\g = \glie$ in \cite{MR1383474}; cf. the second case below.) It is clear that $\rext (-)$   gives an exact functor, as stated.
 
 For the second case, consider an element  of $\lext (N)$ denoted by $[m \otimes x , n]$ for $m , n \in N$ and $x \in \g$, using the splitting of right $\glie$-modules $\lext (N) \cong (N \otimes \g\rreg) \oplus N$. The left and right $\g$-actions are given explicitly for $y, z \in \g$ by:
\begin{eqnarray*}
[m \otimes x , n ] y &=& [my \otimes x + m \otimes  xy , ny ] \\
z [m \otimes x , n ] &=& [n \otimes z , -nz].
\end{eqnarray*}
To check that this defines a $\g$-module structure, one reduces to considering the action on $[0,n]$. 
 The relation $(MLL)$ is clear, since it only involves right multiplication. For $(LML)$ one uses the equalities:
 \begin{eqnarray*} 
 x ([0,n] y ) &=& x [0, ny] = [ny \otimes x , - (ny)x] \\
 (x[0,n])y &=& [n \otimes x, -nx ] y = [ny \otimes x + n \otimes xy , -(nx) y] \\
 (xy)[0,n] &=& [n \otimes xy , -n (xy)] 
 \end{eqnarray*}
to check that $ x ([0,n] y ) =  (x[0,n])y - (xy)[0,n]$, as required. 

Similarly, for $(LLM)$, one has the equalities:
\begin{eqnarray*}
x (y[0,n]) &=& x [n \otimes y , -ny] = [-ny \otimes x, (ny)x]\\
(xy) [0,n] &=& [n \otimes xy, -n (xy)]\\
(x[0,n]) y &=& [n \otimes x , -nx] y = [ny \otimes x + n \otimes xy, -(nx)y]
\end{eqnarray*}
which give $x (y[0,n]) = (xy) [0,n] - (x[0,n]) y$, as required.

It is clear that $\lext (-)$ gives an exact functor, as stated. The duality statement relating the two extensions is established by direct verification. 
\end{proof}

\begin{cor}
\label{cor:ext_classes}
For $\g$ a Leibniz algebra and $N \in \ob \liemod_\glie$, the extensions of Proposition \ref{prop:LP_ses} provide classes
\begin{eqnarray*}
{} [\lext (N)] &\in &\ext^1_{\leibmod_\g}(N\leibsym, (N \otimes \g\rreg)\leibasym)
 \\
{} [\rext(N)]& \in &\ext^1_{\leibmod_\g}(\hom_\kring (\g\rreg, N)\leibsym, N\leibasym )
 \end{eqnarray*}
that are natural in the following sense: if $f : N \rightarrow N'$ is a morphism of $\liemod_\glie$, then 
\begin{eqnarray*}
(f\otimes \g\rreg)_* [\lext(N)] &=& f^* [\lext(N')] \in  \ext^1_{\leibmod_\g}(N\leibsym, ((N' \otimes \g\rreg)\leibasym)
\\
f_* [\rext(N)] &=& (\hom_\kring (\g\rreg, f))^* [\rext(N')] \in \ext^1_{\leibmod_\g}(\hom_\kring (\g\rreg, N)\leibsym, (N')\leibasym ).
\end{eqnarray*}

Moreover, these extensions are related by:
\begin{eqnarray*}
{}[\lext (N) ] &=& \eta^*_N [\rext (N\otimes \g\rreg)]
\\
{} [\rext (N)]&=& (\epsilon_N)_* [\lext (\hom_\kring (\g\rreg, N))]
\end{eqnarray*}
where $\eta_N$ and $\epsilon_N$ are respectively the unit and counit of the adjunction 
$(-\otimes \g\rreg) \dashv \hom_\kring (\g\rreg, -)$.
\end{cor}

\begin{proof}
The first part of the Corollary follows directly from the description of the functors $\lext(-)$ and $\rext(-)$ in Proposition \ref{prop:LP_ses}; in particular, this yields the naturality statement.

The compatibility relations are verified directly, using the identification of $\eta_N$ and $\epsilon_N$ given in Proposition \ref{prop:closed_tensor_lie}.
\end{proof}

\begin{lem}
\label{lem:lext_Uglie}
For $\g$ a Leibniz algebra, there is a natural isomorphism of $\g$-modules:
\[
UL(\g)
\cong 
\lext (U(\glie)).
\]
\end{lem}

\begin{proof}
By construction, $\lext (U(\glie))$ occurs in an extension of the form
\[
0
\rightarrow 
(U(\glie) \otimes \g\rreg)\leibasym
\rightarrow 
\lext (U(\glie))
\rightarrow 
U(\glie)\leibsym
\rightarrow 
0
\]
which splits as right $\glie$-modules. That $\lext (U(\glie))$ is isomorphic to $ UL(\g)$ as a $\g$-module can be checked by comparison with Proposition \ref{prop:lbs_ULg} and the structure of $\ker d_1$ made explicit in the proof of Corollary \ref{cor:ker_d1}.
\end{proof}

 Propositions \ref{prop:lext_deriv_lbs} and \ref{prop:rext_deriv_asinv} below exhibit the key homological properties of $\lext (-)$ and $\rext(-)$ respectively.

\begin{prop}
\label{prop:lext_deriv_lbs}
For $N \in \ob \liemod_\glie$, there are natural isomorphisms:
\[
\mathbb{L}_* \lbs \lext(N) \cong 
\left\{ 
\begin{array}{ll}
N & *= 0 \\
0 & *>0.
\end{array}
\right.
\]
\end{prop}

\begin{proof}
First consider $\lbs \lext (N)$, which is isomorphic to $\lext (N)/ (\lext(N))_0$, by Proposition \ref{prop:leibsym_adjoints}. 
Now, for $n \in N\leibsym \subset \lext (N)$ (using the splitting as $\kring$-modules) and any $x \in \g$, calculating in $\lext (N)$ gives 
\[
xn + nx = n \otimes x \in (N\otimes \g\rreg)\leibasym,
\]
by the definition of $\lext (N)$. One deduces that $(N \otimes \g\rreg) \leibasym = (\lext(N))_0$, giving the natural isomorphism $\lbs \lext (	N) \cong N$ induced by $\lext (N) \twoheadrightarrow N\leibsym$.

It remains to show the vanishing of $(\mathbb{L} \lbs)_* \lext(N) $ for $* >0$. By Lemma \ref{lem:lext_Uglie}, in the case $N= U(\glie)$, $\lext (N)$ is isomorphic to $UL(\g)$ as a $\g$-module. Since this is projective, the result follows  in this case. 

For a general $N$, there exists a short exact sequence of $\glie$-modules of the form:
\[
0
\rightarrow 
K 
\rightarrow 
\bigoplus_{\mathcal{I}} U(\glie)
\rightarrow 
N 
\rightarrow 
0,  
\]
for some indexing set $\mathcal{I}$. Applying the exact functor $\lext(-)$ gives the short exact sequence of $\g$-modules:
\[
0
\rightarrow 
\lext (K)
\rightarrow 
\bigoplus_{\mathcal{I}} \lext (U(\glie))
\rightarrow 
\lext(N) 
\rightarrow 
0.
\]

By the above, the associated long exact sequence has tail:
\[
0
\rightarrow 
(\mathbb{L} \lbs)_1 \lext(N)
\rightarrow 
K 
\rightarrow 
\bigoplus_{\mathcal{I}} U(\glie)
\rightarrow 
N
\rightarrow 
0,
\]
which implies that $(\mathbb{L} \lbs)_1 \lext(N)=0$, since $K \rightarrow \bigoplus_{\mathcal{I}} U(\glie)$ is injective.

This holds for any $\glie$-module $N$; a standard induction using the long exact sequence then shows vanishing in all positive degrees.
\end{proof} 

\begin{rem}
\label{rem:interpret_lext_deriv_lbs}
The result of Proposition \ref{prop:lext_deriv_lbs} can be interpreted conceptually in the derived category of right $\glie$-modules as the natural isomorphism
$
\mathbb{L}\lbs (\lext(N))
\stackrel{\cong}{\rightarrow}
N
$. 
The extension defining $\lext (N)$ gives the distinguished triangle 
\[
\lext (N) 
\rightarrow N^s 
\rightarrow 
(N \otimes \g\rreg)\leibasym [1]
\rightarrow
\]
in the derived category of $\g$-modules, where the morphism $N^s 
\rightarrow 
(N \otimes \g\rreg)\leibasym [1]$ corresponds to the class $[\lext (N)]$.

Then, applying the total derived functor $\mathbb{L}\lbs$  gives the distinguished triangle 
\[
\mathbb{L}\lbs (\lext(N)) \cong N 
\rightarrow 
\mathbb{L}\lbs (N^s)
\rightarrow 
\mathbb{L}\lbs  ((N \otimes \g\rreg)\leibasym) [1]
\rightarrow .
\]
This is split by the canonical morphism $\mathbb{L} \lbs (N^s) \rightarrow N$.  This splitting should be compared with that given in Remark \ref{rem:split_lder}. In particular,  the split surjection $\mathbb{L}\lbs (N^s)
\rightarrow 
\mathbb{L}\lbs  ((N \otimes \g\rreg)\leibasym) [1]$ is induced by $[\lext (N)]$.
\end{rem}

The  result corresponding to Proposition \ref{prop:lext_deriv_lbs} for $\rext (N)$ uses the functor $(-)\asinv$ and its right derived functors. However, the argument reverses the line of reasoning, starting from the analogue of the split distinguished triangle in Remark \ref{rem:interpret_lext_deriv_lbs}.

\begin{prop}
\label{prop:rext_deriv_asinv}
For $N\in \ob \liemod_\glie$, there are natural isomorphisms
\[
\mathbb{R}^* (-)\asinv \rext (N) 
\cong 
\left\{
\begin{array}{ll}
N & *=0 \\
0 & *>0.
\end{array}
\right.
\]
\end{prop}

\begin{proof}
Using the explicit description of $(-)\asinv$ given in Proposition \ref{prop:antisym_adjunction}, one checks that $(\rext(N))\asinv=N$. 

For the higher derived functors, use the identification of the total right derived functor 
\[
\mathbb{R} (-)\asinv M = \hom_{UL(\g)} (W_\bullet (\g), M)
\]
for $M \in \ob \leibmod_\g$, given by Corollary \ref{cor:total_right}.

The defining short exact sequence for $\rext (N)$ therefore gives the short exact sequence of complexes:
\[
\xymatrix{
0
\ar[r]
&
\hom_{UL(\g)} (W_\bullet(\g), N\leibasym)
\ar[r]
\ar[d]^\cong
&
\hom_{UL(\g)} (W_\bullet(\g), \rext(N))
\ar[r]
\ar[d]^{\cong}
&
\hom_{UL(\g)} (W_\bullet(\g), \hom_\kring (\g\rreg, N)\leibsym)
\ar[r]
\ar[d]^\cong
&
0
\\
0
\ar[r]
&
\hom_\kring (\g^{\otimes \bullet} , N) 
\ar[r]
&
\hom_\kring (\g^{\otimes \bullet} , \rext(N))
\ar[r]
&
 \hom_\kring (\g^{\otimes \bullet +1} , N) 
\ar[r]
&
0,
}
\]
where the isomorphism $\hom_\kring (\g^{\otimes \bullet}, \hom_\kring (\g\rreg, N))\cong \hom_\kring (\g^{\otimes \bullet +1} , N)$ has been used for the bottom right hand term. 

The differential of the complex $\hom_\kring (\g^{\otimes \bullet} , N) $ takes into account that the $\g$-module $N\leibasym$ is antisymmetric and the differential of $\hom_{\leibmod_\g} (W_\bullet(\g), \hom_\kring (\g\rreg, N)\leibsym)$ takes into account that this complex arose from the symmetric coefficients $\hom_\kring (\g\rreg, N)\leibsym$.

The proof of \cite[Lemma 1.4(b)]{FW} shows that the outer complexes are isomorphic, up to the shift of degree. More precisely, the argument shows that the connecting isomorphism in the associated long exact sequence for cohomology is an isomorphism in positive degrees. The result follows. 
\end{proof}

\begin{rem}
\label{rem:rder_rext}
The class $[\rext (N)]$ is represented by a morphism $\hom_\kring (\g \rreg, N) \leibsym [-1] \rightarrow  N\leibasym$ in the derived category of $\g$-modules. Proposition \ref{prop:rext_deriv_asinv} shows that, on applying $\mathbb{R} (-)\asinv$, one obtains the split inclusion in the derived category of right $\glie$-modules:
\[
\mathbb{R} (-)\asinv (\hom_\kring (\g \rreg, N) \leibsym )[-1]
\rightarrow 
\mathbb{R} (-)\asinv (N\leibasym)
\]
corresponding to the splitting of Remark \ref{rem:split_rder}.
\end{rem}

\section{Relating $\ext$ as $\g$-modules and as $\glie$-modules}
\label{sect:ext}

Let $\g$ be a Leibniz algebra. The symmetric and antisymmetric module functors
\[
(-)\leibsym, \ (-)\leibasym 
: 
\liemod_\glie
\rightrightarrows 
\leibmod_\g
\]
are exact and have both left and right adjoints (see Section \ref{sect:sym_asym}). In particular, one has the following:

\begin{lem}
\label{lem:ext_leib(a)sym}
For $\g$ a Leibniz algebra and $ N_1, N_2$ $\glie$-modules, 
\begin{enumerate}
\item 
$(-)\leibsym$ induces a natural split monomorphism
$\ext^*_{\liemod_\glie}  (N_1, N_2) 
\hookrightarrow 
\ext^*_{\leibmod_\g}  (N_1\leibsym, N_2\leibsym)$; 
\item 
$(-)\leibasym$ induces a natural split monomorphism
$\ext^*_{\liemod_\glie}  (N_1, N_2) 
\hookrightarrow 
\ext^*_{\leibmod_\g}  (N_1\leibasym, N_2\leibasym)$.
\end{enumerate}
\end{lem}

\begin{proof}
The natural morphisms are induced by the respective exact functors $(-)\leibsym$ and $(-)\leibasym$. The restriction functor $(-)\rreg : \leibmod_\g \rightarrow \liemod_\glie$ of Proposition \ref{prop:right-restriction} is exact and induces the natural retract.
\end{proof}

The main result of the paper  refines this by identifying the complement:

\begin{thm}
\label{thm:reduction}
For $\g$ a Leibniz algebra and $ N_1, N_2$ $\glie$-modules, 
\begin{enumerate}
\item 
there is a natural isomorphism
\[
\ext^* _{\leibmod_\g}(N_1\leibsym, N_2 \leibsym )
\cong 
\ext^*_{\liemod_\glie} (N_1 , N_2) 
\oplus
\ext^{*-1} _{\leibmod_\g}((N_1 \otimes \g\rreg)\leibasym, N_2 \leibsym )
\]
and the inclusion $\ext^{*-1} _{\leibmod_\g}((N_1 \otimes \g\rreg)\leibasym,  N_2  \leibsym )
\hookrightarrow \ext^* _{\leibmod_\g}( N_1 \leibsym,  N_2  \leibsym )
$ is given by Yoneda product with the class  $[\lext (N_1)]$;
\item 
there is a natural isomorphism
\[
\ext^* _{\leibmod_\g}( N_1 \leibasym,  N_2  \leibasym )
\cong 
\ext^*_{\liemod_\glie} (N_1 , N_2) 
\oplus
\ext^{*-1} _{\leibmod_\g}( N_1 \leibasym, \hom_\kring (\g\rreg, N_2) \leibsym )
\]
and the inclusion $\ext^{*-1} _{\leibmod_\g}( N_1 \leibasym, \hom_\kring (\g\rreg, N_2) \leibsym )
\hookrightarrow \ext^* _{\leibmod_\g}( N_1 \leibasym,  N_2  \leibasym )
$ is given by Yoneda product with the class  $[\rext(N_2)]$.
\end{enumerate}
\end{thm}

\begin{rem}
The rôle of the extension classes $[\lext (N_1)]$ and $[\rext (N_2)]$ in Theorem \ref{thm:reduction} was inspired by the usage of the extension $\rext (-)$ by Loday and Pirashvili (see 
 \cite[Remark 3.2]{MR1383474}).
\end{rem}

Theorem \ref{thm:reduction} is  a consequence of the following, which highlights the homological rôle of the functors $\lext (-)$ and $\rext (-)$. 

\begin{prop}
\label{prop:alternative_lext_rext}
For $\g$ a Leibniz algebra and $ N_1, N_2$ $\glie$-modules, 
\begin{enumerate}
\item 
the composite 
\[
\ext^*_{\liemod_\glie} (N_1, N_2) 
\rightarrow 
\ext^*_{\leibmod_\g} (N_1\leibsym, N_2 \leibsym) 
\rightarrow 
\ext^*_{\leibmod_\g} (\lext (N_1), N_2 \leibsym) 
\]
is an isomorphism, where the first map is induced by the functor $(-)\leibsym$ and the second by the canonical surjection $\lext (N_1) \twoheadrightarrow N_1\leibsym$;
\item 
the composite 
\[
\ext^*_{\liemod_\glie} (N_1, N_2) 
\rightarrow 
\ext^*_{\leibmod_\g} (N_1\leibasym, N_2 \leibasym) 
\rightarrow 
\ext^*_{\leibmod_\g} (N_1\leibasym, \rext (N_2)) 
\]
is an isomorphism, where the first map is induced by the functor $(-)\leibasym$ and the second by the canonical inclusion $N_2 \leibasym \hookrightarrow \rext (N_2)$.
\end{enumerate}
\end{prop}

\begin{proof}
For $\g$ a Leibniz algebra, $M$ a $\g$-module and $N$ a $\glie$-module, there are natural, convergent first quadrant cohomological Grothendieck spectral sequences: 
\begin{eqnarray*}
\ext^p_{\liemod_\glie} ( \mathbb{L}_q\lbs M, N) 
&\Rightarrow &
\ext^{p+q}_{\leibmod_\g}  (M,  N\leibsym)
\\
\ext^p_{\liemod_\glie} ( N , \mathbb{R}^q(-)\asinv M) ) 
&\Rightarrow &
\ext^{p+q}_{\leibmod_\g}  ( N\leibasym, M).
\end{eqnarray*}

For the first statement, take $M= \lext (N_1)$ and $N=N_2$. Proposition \ref{prop:lext_deriv_lbs} gives that 
$\mathbb{L}_q\lbs (\lext (N_1))=0$ for $q>0$ and is $N_1$ for $q=0$. It follows that the first spectral sequence collapses at the $E_2$-page and the required isomorphism is the edge isomorphism of the spectral sequence. 

For the second statement, take $M = \rext (N_2)$ and $N= N_1$. Proposition \ref{prop:rext_deriv_asinv} gives that $\mathbb{R}^q (-)\asinv (\rext (N_2))=0$ for $q>0$ and is $N_2$ for $q=0$. Thus the second spectral sequence collapses at the $E_2$-page and, once again, the required isomorphism is given by the edge isomorphism.
\end{proof}

\begin{proof}[Proof of Theorem \ref{thm:reduction}]
This follows  from Proposition \ref{prop:alternative_lext_rext}. The argument is given for the first case; the second is proved similarly. 

Consider the short exact sequence $0\rightarrow (N_1 \otimes \g \rreg)\leibasym \rightarrow \lext (N_1) \rightarrow N_1 \leibsym \rightarrow 0$ in $\leibmod_\g$. Applying $\ext^*_{\leibmod_\g} (-, N_2\leibsym)$ gives 
\[
\xymatrix{
&&&
\ext^*_{\liemod_\glie} (N_1, N_2) 
\ar[d]_\cong 
\ar[ld]
\\
\ldots
\ar[r]
&
\ext^{*-1}_{\leibmod_\g} ((N_1 \otimes \g\rreg)\leibasym, N_2 \leibsym) 
\ar[r]
&
\ext^*_{\leibmod_\g} (N_1\leibsym, N_2 \leibsym) 
\ar[r]
&
\ext^*_{\leibmod_\g} (\lext(N_1), N_2 \leibsym)
\ar[r]
& \ldots
 }
\]
in which the horizontal row is the associated long exact sequence and the commutative triangle is provided by Proposition \ref{prop:alternative_lext_rext}. This yields the required splitting.
\end{proof}

\begin{rem}
A conceptual proof of Theorem \ref{thm:reduction} can  be given  working at the level of the appropriate derived categories, as sketched below. 

For the first statement, as in Remark \ref{rem:interpret_lext_deriv_lbs}, applying the total derived functor $\mathbb{L} \lbs $ to the  distinguished triangle in the derived category of $\leibmod_\g$ 
associated to $\lext (N_1)$: 
\[
\lext (N_1) 
\rightarrow 
N_1\leibsym 
\rightarrow 
(N_1 \otimes \g\rreg)\leibasym
[1]
\rightarrow 
\]
gives the split distinguished triangle
\[
N_1
\rightarrow 
\mathbb{L}\lbs(N_1\leibsym )
\rightarrow 
\mathbb{L}\lbs \big((N_1 \otimes \g\rreg)\leibasym\big) 
[1]
\rightarrow 
\]
in the derived category of $\liemod_\glie$. Calculating morphisms $[-, N_2[*]]$ in this derived category then gives the result, using that $\mathbb{L}\lbs$ is left adjoint to $(-)\leibsym$.  
 
 The approach to the second statement is similar, this time appealing to Remark \ref{rem:rder_rext}. Namely, there is a distinguished triangle in the derived category of $\leibmod_\g$ associated to $\rext (N_2)$:
 \[
 \hom_\kring (\g\rreg, N_s)\leibsym [-1]
 \rightarrow 
 N_2\leibasym 
 \rightarrow 
 \rext (N_2) 
 \rightarrow 
 \]
where the left hand morphism represents $[\rext (N_2)]$. Applying $\mathbb{R} (-)\asinv$ gives the distinguished triangle: 
\[
\mathbb{R} (-)\asinv\big(\hom_\kring (\g \rreg, N_2) \leibsym \big)[-1]
 \rightarrow 
 \mathbb{R} (-)\asinv(N_2  \leibasym)
  \rightarrow
 N_2 \rightarrow, 
\]
since $\mathbb{R} (-)\asinv(\rext (N_2)) \cong N_2$ in the derived category of $\liemod_\glie$, by Proposition \ref{prop:rext_deriv_asinv}. This distinguished triangle is split. The result follows by applying $[N_1, -[*]]$ using that $\mathbb{R} (-) \asinv$ is right adjoint to $(-)\leibasym$.
\end{rem}

\subsection{Example applications}

Recall that Lie algebra cohomology $H^* (\glie; N)$ with coefficients in a $\glie$-module $N$ is naturally isomorphic to $\ext^* _{\liemod_\glie} (\kring, N)$. Thus Theorem \ref{thm:reduction} implies:

\begin{cor}
\label{cor:reduction_one_trivial}
For $\g$ a Leibniz algebra and $N$ a $\glie$-module, there are natural isomorphisms:
\begin{eqnarray*}
\ext^* _{\leibmod_\g}(\kring, N \leibsym )
&\cong & 
H^* (\glie; N) 
\oplus
\ext^{*-1} _{\leibmod_\g}((\g\rreg)\leibasym, N \leibsym )
\\
\ext^* _{\leibmod_\g}(N \leibsym, \kring )
&\cong & 
H^* (\glie; N^\sharp)  
\oplus
\ext^{*-1} _{\leibmod_\g}((N \otimes \g\rreg)\leibasym, \kring )
\\
\ext^* _{\leibmod_\g}(\kring, N \leibasym )
&\cong &  
 H^* (\glie; N)
\oplus
\ext^{*-1} _{\leibmod_\g}(\kring , \hom_\kring (\g\rreg, N) \leibsym )
\\
\ext^* _{\leibmod_\g}(N \leibasym, \kring )
&\cong &
H^* (\glie; N^\sharp)  
\oplus
\ext^{*-1} _{\leibmod_\g}(N\leibasym,  ((\g\rreg)^\sharp) \leibsym),
\end{eqnarray*}
where $\kring$ is the trivial $\g$-module.

Hence there are natural isomorphisms:
\begin{eqnarray*}
\ext^* _{\leibmod_\g}(N \leibsym, \kring )
&\cong & 
H^* (\glie; N^\sharp)  
\oplus
H^{*-1} (\glie; (N \otimes \g\rreg)^\sharp)
\oplus
\ext^{*-2} _{\leibmod_\g}((N \otimes \g\rreg)\leibasym ,  ((\g\rreg)^\sharp) \leibsym)
\\
\ext^* _{\leibmod_\g}(\kring, N \leibasym )
&\cong &  
 H^* (\glie; N)
\oplus
H^{*-1} (\glie;  \hom_\kring (\g\rreg, N))
\oplus 
\ext^{*-2} _{\leibmod_\g}((\g\rreg)\leibasym, \hom_\kring (\g\rreg, N) \leibsym ).
\end{eqnarray*}
\end{cor}

\begin{rem}
There are duality relations between the above statements (cf. Corollary \ref{cor:duality_ext}). Moreover, there are related duality  {\em isomorphisms} when $\g$ is finite-dimensional and one restricts to working in the derived category of finite-dimensional modules (cf. Remark \ref{rem:duality_not_iso}).
\end{rem}

\begin{exam}
Let $\g$ be a Leibniz algebra over a field of characteristic zero  such that $\glie$ is semisimple; thus Whitehead's lemmas imply that $H^1 (\glie ; N) = H^2(\glie; N) =0$ for any finite-dimensional $\glie$-module $N$. Moreover, Weyl's theorem implies that the category of finite-dimensional $\glie$-modules is semisimple.

Thus, in cohomological dimension $\leq 2$, Theorem \ref{thm:reduction} and its Corollary \ref{cor:reduction_one_trivial} give generalizations of the results in \cite{MR1383474} and \cite{MW}.
\end{exam}

\section{The case of $\tor$}
\label{sect:tor}

This Section explains the counterparts of the  splitting results of Section \ref{sect:ext} for $\tor$. Throughout, $M$ will be taken to be a right $\glie$-module and $N$ a left $\glie$-module. Adapting Notation \ref{nota:kurd} to $U(\glie)$-modules using the conjugation $\chi$, one has the associated left $\glie$-module $M\kurd$ and the right $\glie$-module ${\kurd}N$. 

The analogue of considering both coefficients to be symmetric or both coefficients to be antisymmetric is given 
by  the following bifunctors of $M$ and $N$: 
\begin{eqnarray*}
M^{d_1} \otimes_{UL(\g)} {}^{d_1} N \\
M^{d_0} \otimes_{UL(\g)} {}^{d_0} N 
\end{eqnarray*}
together with their derived functors, which give the homological analogues of those considered in Theorem \ref{thm:reduction}. In this context, the Kurdiani involution allows reduction to the first case, as established below by Proposition \ref{prop:kurd_tor}.

This relies upon the  following Lemma, which is closely related to Corollary \ref{cor:left_right_Uglie} and is proved using Proposition \ref{prop:compat_chiL}:

\begin{lem}
\label{lem:kurd_glie}
For $N$ a left $\glie$-module, there are natural isomorphisms of right $UL(\g)$-modules:
\begin{eqnarray*}
{}\kurd({}^{d_1} N) &\cong & ({}\kurd N) ^{d_0} \\
{}\kurd({}^{d_0} N) &\cong & ({}\kurd N) ^{d_1}.
\end{eqnarray*}
\end{lem} 

\begin{rem}
As in Section \ref{sect:sym_asym}, the right $UL(\g)$-module $M^{d_1}$ identifies with $M\leibsym$. By  Lemma \ref{lem:kurd_glie}, 
there is a natural isomorphism 
\[
{}^{d_1} N
\cong 
\big( ({}\kurd N)\leibasym \big)\kurd.
\]
When working with both left and right $UL(\g)$-modules, the notation $(-)^{d_1}$ and ${}^{d_1} (-)$ is more appropriate.
\end{rem}

\begin{prop}
\label{prop:kurd_tor}
The Kurdiani involution induces  natural isomorphisms
\begin{eqnarray*}
M^{d_1} \otimes _{UL(\g)} {}^{d_1} N 
&\cong &
({}\kurd N)^{d_0} \otimes_{UL(\g)} {}^{d_0} (M\kurd)
\\
\tor^{UL(\g)} _* (M^{d_1} , {}^{d_1} N) 
&\cong& 
\tor^{UL(\g)} _* ( ( {}\kurd N)^{d_0} , {}^{d_0} (M\kurd)).
\end{eqnarray*}
\end{prop}

\begin{proof}
The first isomorphism is a consequence of Lemma \ref{lem:kurd_glie} and the second follows on passage to derived functors.
\end{proof}

\begin{rem}
Proposition \ref{prop:kurd_tor} does not have a precise counterpart when considering $\ext$. In the latter context one has to use the duality functor $(-)^\sharp$ in place of the equivalence of categories arising from the Kurdiani involution. This is not an equivalence of categories and the induced morphism on $\ext$  (cf. Corollary \ref{cor:duality_ext}) is not in general an isomorphism.
\end{rem}

\begin{thm}
\label{thm:tor}
For $\g$ a Leibniz algebra, $M$ a right $\glie$-module and $N$ a left $\glie$-module, there is a natural isomorphism:
\[
\tor^{UL(\g)}_* (M^{d_1}, {}^{d_1}N) 
\cong 
\tor^{U(\glie)}_* (M, N)
\oplus 
 \tor^{UL(\g)}_{*-1} ((M\otimes \g\rreg)^{d_0}, {}^{d_1}N) 
\] 
where the morphism $\tor^{UL(\g)}_* (M^{d_1}, {}^{d_1}N)  \rightarrow \tor^{UL(\g)}_{*-1} ((M\otimes \g\rreg)^{d_0}, {}^{d_1}N) $ is induced by cap product with the class $[\lext (M)]$.
\end{thm}

\begin{proof}
The proof is analogous to that of Theorem \ref{thm:reduction}, {\em mutatis mutandis}.
\end{proof}

\begin{rem}
Theorem \ref{thm:tor} gives a generalization of Example \ref{exam:HL_hom_splitting} to $\tor$ in the same way that Theorem \ref{thm:reduction} generalizes Example \ref{exam:splitting_HL_cohom_antisymmetric} to $\ext$.
\end{rem}



\begin{thebibliography}{MW21}

\bibitem[CE99]{MR1731415}
Henri Cartan and Samuel Eilenberg, \emph{Homological algebra}, Princeton
  Landmarks in Mathematics, Princeton University Press, Princeton, NJ, 1999,
  With an appendix by David A. Buchsbaum, Reprint of the 1956 original.
  \MR{1731415}

\bibitem[Fel19]{Feld}
J\"{o}rg Feldvoss, \emph{Leibniz algebras as non-associative algebras},
  Nonassociative mathematics and its applications, Contemp. Math., vol. 721,
  Amer. Math. Soc., Providence, RI, 2019, pp.~115--149. \MR{3898506}

\bibitem[FW21]{FW}
J\"{o}rg Feldvoss and Friedrich Wagemann, \emph{On {L}eibniz cohomology}, J.
  Algebra \textbf{569} (2021), 276--317. \MR{4187237}

\bibitem[Kur99]{MR1709257}
Revaz Kurdiani, \emph{Cohomology of {L}ie algebras in the tensor category of
  linear maps}, Comm. Algebra \textbf{27} (1999), no.~10, 5033--5048.
  \MR{1709257}

\bibitem[LP93]{LP}
Jean-Louis Loday and Teimuraz Pirashvili, \emph{Universal enveloping algebras
  of {L}eibniz algebras and (co)homology}, Math. Ann. \textbf{296} (1993),
  no.~1, 139--158. \MR{1213376}

\bibitem[LP96]{MR1383474}
\bysame, \emph{Leibniz representations of {L}ie algebras}, J. Algebra
  \textbf{181} (1996), no.~2, 414--425. \MR{1383474}

\bibitem[MW21]{MW}
Jean Mugniery and Friedrich Wagemann, \emph{Ext groups in the category of
  bimodules over a simple {L}eibniz algebra}, J. Pure Appl. Algebra
  \textbf{225} (2021), no.~6, 106637, 15. \MR{4173997}

\end{thebibliography}

\providecommand{\bysame}{\leavevmode\hbox to3em{\hrulefill}\thinspace}
\providecommand{\MR}{\relax\ifhmode\unskip\space\fi MR }
\providecommand{\MRhref}[2]{%
  \href{http://www.ams.org/mathscinet-getitem?mr=#1}{#2}
}
\providecommand{\href}[2]{#2}

\end{document}